\documentclass[reqno]{amsart}
\usepackage{amssymb}
\usepackage{amsmath}
\usepackage[mathscr]{euscript}
\usepackage{pgf,tikz} 
\usetikzlibrary{arrows}
\usepackage[small]{caption}
\usepackage{psfrag}
\usepackage{hyperref}
\usepackage{subfigure}

\makeatletter
\@addtoreset{equation}{section}
\makeatother

\newtheorem{theorem}{Theorem}[section]
\newtheorem{lemma}[theorem]{Lemma}
\newtheorem{proposition}[theorem]{Proposition}
\newtheorem{corollary}[theorem]{Corollary}

\newtheorem{remark}[theorem]{Remark}

\newcommand{\bb}[1]{{\mathbb #1}}

\newcommand{\supp}{\mathop{\rm supp}}

\begin{document}

\title[Scaling limit for a long-range divisible sandpile]{Scaling limit for a long-range divisible sandpile}

\author[Susana Fr\'ometa and Milton Jara]{Susana Fr\'ometa and Milton Jara}

\begin{abstract}
We study the scaling limit of a divisible sandpile model associated with a truncated $\alpha$-stable random walk. We prove that the limiting distribution is related to an obstacle problem for a truncated fractional Laplacian. We also provide, as a fundamental tool, precise asymptotic expansions for the corresponding rescaled discrete Green's functions. In particular, the convergence rate of these Green's functions to its continuous counterpart is derived.
\end{abstract}

\thanks{S.F. and M.J. were partially supported by CNPq and FAPERJ}

\address{\noindent Susana Fr\'ometa: IMPA, Estrada Dona Castorina 110, CEP 22460 Rio de
   Janeiro, Brasil and UFRGS, Campus do Vale, bairro Agronomia, Porto Alegre, Brasil \newline e-mail: \rm
   \texttt{susana@impa.br}}

\address{\noindent Milton Jara: IMPA, Estrada Dona Castorina 110, CEP 22460 Rio de
   Janeiro, Brasil \newline e-mail: \rm
   \texttt{mjara@impa.br}}



\maketitle

\newcommand{\<}{\langle}
\renewcommand{\>}{\rangle}
\renewcommand{\Cap}{{\rm cap}}

\section{Introduction}
In this paper we study the divisible sandpile model, which is a continuous version of the Abelian sandpile introduced in 1987 by Bak, Tang and Weisenfeld \cite{BTW} as an example of a dynamical system displaying self-organized criticality. 

To our knowledge, the divisible sandpile was introduced by Levine and Peres in \cite{Levine1}. The main difference between the divisible sandpile and the Abelian sandpile is that instead of discrete particles of unit mass, each site can contain a continuous amount of mass. At the start there is an initial density of mass distributed on the lattice $\mathbb{Z}^d$. A lattice site is said to be a {\em full site} if it has mass at least $1$. At each time step, each full site is {\em toppled} by keeping mass $1$ for itself and distributing the excess mass among the lattice proportionally to the step distribution of a certain transition probability. As time goes to infinity the mass approaches a final distribution in which each site has mass less than or equal to $1$.

In \cite{Levine1} and \cite{Levine2}, Levine and Peres studied the scaling limit of this final distribution for the divisible sandpile in which the excess mass is split in every toppling equally among neighbors. In other words, they study the divisible sandpile model associated with the transition probability of a symmetric, simple random walk on $\mathbb{Z}^d$. The proof is based within the framework of simple random walks. The corresponding scaling limit is related to the obstacle problem for the classical Laplacian operator. In \cite{cyrille} Lucas studied a divisible sandpile model that he calls the {\em unfair divisible sandpile}, in which the toppling procedure distributes mass to neighbors proportionally to a nearest neighbors, drifted random walk. The corresponding scaling limit of the final distribution is related to a so-called {\em true heat ball}.

\begin{figure}[ht]
	\centering
	\subfigure[$3400$ iterations]{\includegraphics[width=0.35\textwidth]{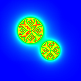}}\hspace{0.1\textwidth}
	\subfigure[$26500$ iterations]{\includegraphics[width=0.35\textwidth]{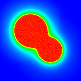}}
	\caption{Truncated $\alpha$-stable sandpile starting with masses $700$ and $1300$ in two different sites. In each picture we use a scale of color that indicates the amount of mass at each site.}\label{figure1}
\end{figure}

The so-called {\em internal diffusion limited aggregation} model (iDLA) is a stochastic version of the Abelian sandpile introduced by Lawler, Bramson and Griffeath in \cite{lawler1992}. In the mentioned paper, the authors studied the scaling limit of the random set of final occupied sites starting with $m$ particles at the origin, which happens to be a ball with volume $m$ in $\mathbb R^d$. In \cite{Levine2} Levine and Peres also studied the scaling limit of the iDLA but starting from an arbitrary distribution of mass, they proved that considering the same initial distribution in the divisible sandpile and in the iDLA, the scaling limit of the final set of occupied sites in both models happens to be the same. In \cite{cyrille}, a drifted version of the iDLA was also studied and the scaling limit was also proved to coincide with the scaling limit of the drifted divisible sandpile.   

Regarding the Abelian sandpile model some recent advances correspond to Pegden and Smart; they proved in \cite{smart} that the Abelian sandpile converges in the weak-$*$ topology to a limit characterized by an elliptic obstacle problem. Later in \cite{LevineSmart}, Levine, Pegden, and Smart identified Apollonian structures in the scaling limit of the Abelian sandpile. Structures of the limit of the Abelian sandpile were previously observed in \cite{SadhuDhar}.     

Our aim is to study a family of divisible sandpiles, which we call {\em truncated, $\alpha$-stable} divisible sandpiles, that distribute mass not only to nearest neighbors. We present a scaling limit for the final distribution of a sequence of divisible sandpiles on which the excess mass is distributed proportionally to the transition probability of a truncated $\alpha$-stable random walk. Figure \ref{figure1} represents a simulation of the truncated $\alpha$-stable sandpile starting with mass in two different lattice sites.

Let $\rho$ be a nonnegative function on $\mathbb{R}^2$ with compact support. We will use this function $\rho$ to choose the initial mass density of the truncated $\alpha$-stable sandpile. As the lattice spacing goes to zero, we run the sandpile with an initial distribution given by $\rho$ restricted to the lattice. We follow the approach of \cite{Levine1} and \cite{Levine2} that uses the odometer function as a fundamental tool for the study of sandpiles models. In Theorem \ref{odoconvergence} we obtain that the odometer function related to the $\alpha$-stable sandpile converges uniformly to its continuous version. We also obtain, as stated in Theorem~\ref{convergence-distribution}, that the final distribution of the rescaled sandpile converges and that its limit is related to an obstacle problem for the truncated fractional Laplacian. Obstacle problems are a current subject of interest in the literature; see, for example, \cite{CaffarelliObstacle}, \cite{Friedman2010}. Nevertheless, to our knowledge the obstacle problem we need to deal with in this work has not been considered in the literature. The approach of \cite{SilvestreTesis} can be adapted to our situation, although the necessary changes are not straightforward. The main source of trouble comes from the fact that the truncated Laplacian is not self-similar. A more detailed discussion can be found in Section \ref{CONT}.

Green's function estimates for the truncated $\alpha$-stable random walk are a key tool in this work. In the context of simple random walks, such estimates are well known. A comprehensive reference, which also shows the power of these estimates as a tool in a variety of contexts, is Lawler's book \cite{LawlerIntersections}. In \cite{LeGall}, Le Gall and Rosen obtained precise estimates for Green's functions of random walks in the domain of attraction of stable laws. 

In the specific context of truncated $\alpha$-stable random walks with $\alpha \in (1,2)$ in dimension $d=2$, we present in Theorem \ref{GMntoGM} a precise asymptotic expansion for the corresponding rescaled discrete Green's functions. In particular, the convergence rate of these Green's functions to its continuous counterpart is derived. It turns out that the estimates of \cite{LeGall} applied in this context are not enough to prove Theorem \ref{odoconvergence} and more precise estimates are needed. 

This asymptotic expansion is the most technical part of this work. In particular, the optimal convergence rate $n^{-\alpha}$ is obtained only after the introduction of a non-trivial, nearest-neighbor correction to the transition probability of the truncated $\alpha$-stable sandpile. Without this correction term, the convergence rate can be checked to be $n^{\alpha-2}$. 

The asymptotic expansion of the Green's function stated in Theorem \ref{GMntoGM} is of independent interest and it can have various applications in numerical methods related to pseudodifferential operators like the fractional Laplacian. What Theorem \ref{GMntoGM} is saying is that if we take $k_1=0$ in \eqref{probdim1}, we obtain a discrete approximation of the truncated Laplacian \eqref{truncfractlapl} for which the associated Green's function converges at rate $n^{-(2-\alpha)}$, which is   far from the optimal rate $n^{-\alpha}$ and gets worse as $\alpha \to 2$. Introducing the calibration constant $k_1$ we obtain a discrete approximation of \eqref{truncfractlapl} for which the Green's function converges at the rate $n^{-\alpha}$, which can be checked to be optimal by scaling. 

The convergence of the final distribution of particles for noncalibrated, truncated $\alpha$-stable sandpiles remains open. We conjecture that calibration is needed in order to obtain the strong convergence of the odometer function stated in Theorem~\ref{odoconvergence}, but that the weak-$*$ convergence of the final distribution stated in Theorem~\ref{convergence-distribution} holds without the calibration.

\subsection{Notations and results}

The next few paragraphs are devoted to describe the random walk we use in the toppling procedure of the truncated $\alpha$-stable divisible sandpile model.

Let us introduce the following notations:

For $R>0$, we denote $B_R=\{x\in\mathbb{R}^2;|x|<R\}$. For a domain $A$ in $\mathbb{R}^2$, we denote $A^{::}=A\cap\mathbb{Z}^2$. We define the discrete boundary of the domain $A$, as 
$$\partial A^{::}=\{x\in\mathbb{Z}^2;[x-\tfrac{1}{2},x+\tfrac{1}{2}]^2 ~/\!\!\!\!\!\subset A\text{~and~}[x-\tfrac{1}{2},x+\tfrac{1}{2}]^2\cap A\neq\emptyset\}.$$

We fix a parameter $\alpha\in(1,2)$, and two positive parameters $r<M$. We define $A_{r,M}:=B_M\setminus B_r$. Now we will consider a function $F_1=F_{1,r,M}:\mathbb{Z}^2\rightarrow\mathbb{R}$, which is essentially the indicator of the set $A_{r,M}^{::}$, but for computational convenience, our function $F_1$ takes specific values in $\partial A_{r,M}^{::}$. We define $F_1(y)$ as the proportion of the square $[y-\frac{1}{2},y+\frac{1}{2}]^2$ that is inside the set $A_{r,M}$, that is,
$$F_1(y)=\mu([y-\tfrac{1}{2},y+\tfrac{1}{2}]^2\cap A_{r,M}),\text{~for~}y\in\mathbb{Z}^2,$$
where $\mu$ represents the Lebesgue measure in $\mathbb{R}^2$. Notice that $F_1$ is equal to $1$ in $A_{r,M}^{::}\setminus\partial A_{r,M}^{::}$, and equal to $0$ outside $A_{r,M}^{::}\cup\partial A_{r,M}^{::}$.

We define the probability jump $p:\mathbb{Z}^2\rightarrow [0,1]$ as
\begin{equation}\label{probdim1}
	p(y)=c_1\Big(k_1\textbf{1}_{\{|y|=1\}}+\frac{F_1(y)}{|y|^{2+\alpha}}\Big)\text{~~for~~}y\in\mathbb{Z}^2\setminus\{0\}
\end{equation}
and $p(0)=0$, where $c_1$ is the normalizing constant and $k_1$ is a positive constant, which depends on $\alpha$, $r$, and $M$, defined as
\begin{equation}\label{kdim1}
	k_1=\frac{1}{4}\Big(\int_{B_M}\frac{1}{|y|^\alpha}dy-\!\!\!\!\sum_{y\in A_{r,M}^{::}\cup\partial A_{r,M}^{::}}\frac{F_1(y)}{|y|^\alpha}\Big).
\end{equation}
We choose $r$ and $M$ in order to have $k_1>0$; see Lemma \ref{knpositive}. Then the probability~\eqref{probdim1} is well defined. 

In order to obtain the desired convergence of the Green's function, instead of working with a pure truncated $\alpha$-stable random walk, we add a simple random walk multiplied by a carefully chosen constant. This is the constant $k_1$ defined in~\eqref{kdim1}. This strategy leads us to cancellations that allow us to obtain convergence of order $n^{-\alpha}$ as Theorem \ref{GMntoGM} states. In particular, the convergence stated in Theorem \ref{GMntoGM} does not hold without this correction.

Now we run in $\mathbb{Z}^2$ our sandpile model starting with initial distribution $\rho$. When time goes to infinity, in Proposition \ref{existence-odometer} below we will prove that if every full site is toppled infinitely often, then the mass converges to a limiting distribution on which each site has mass taking values between $0$ and $1$.

Two individual topplings in the truncated $\alpha$-stable sandpile do not commute, but the sandpile is \textit{Abelian} in the sense that the limiting distribution does not depend on the ordering of topplings. This is the subject of Proposition \ref{abelianprop}.

In the truncated $\alpha$-stable divisible sandpile, as in the classical sandpile model studied in \cite{Levine1} and \cite{Levine2}, we use the \textit{odometer function} as the fundamental tool to identify the limiting mass distribution. The odometer function measures the total amount of mass emitted from the point $x\in\mathbb{Z}^2$, counted with repetitions:
\begin{equation}\label{odometerdim1}
	u(x)=\text{total mass emitted from } x.
\end{equation}

Also as a consequence of the Abelian property, the odometer function does not depend on the order of the topplings. The total mass received by site $x$ from other lattice points throughout the toppling procedure is equal to 
$$	\sum_{y\in B_M^{::}\cup\partial B_M^{::}}, \!\!\!\!\!\!p(y)u(x+y),$$
hence,
\begin{equation}\label{L1(u)}
	L_{M,1} u(x)=\nu(x)-\rho(x),
\end{equation}
where $\rho$ and $\nu$ are, respectively, the initial and final distribution of mass, and $L_{M,1}$ is the discrete Laplacian related to $p$ given by
\begin{equation}\label{lapldim1}
	L_{M,1}f(x)= \!\!\!\!\!\!\sum_{y\in B_M^{::}\cup\partial B_M^{::}}  \!\!\!\!\!\! p(y)(f(x+y)-f(x)) 
\end{equation}
for a function $f$ in $\mathbb{Z}^2$.

We say that a function $f$ is superharmonic at a point $x\in\mathbb{Z}^2$ with respect to $L_{M,1}$ if $L_{M,1}f(x)\leq0$.  

To study the limiting distribution we will use the following approach. First we construct a function $\gamma$ on $\mathbb{Z}^2$ whose discrete Laplacian is equal to $\rho-1$; this will be done in Section \ref{obstacle-function}. Assuming we have this function, we consider its least superharmonic majorant
\begin{equation}\label{majorantdim1}
	s(x)=\inf\{f(x);L_{M,1}f\leq0 \text{ on } \mathbb{Z}^2 \text{ and } f\geq\gamma\}.
\end{equation}
The majorant $s$ is the solution of the discrete obstacle problem for the operator $L_{M,1}$ with obstacle $\gamma$.

The following lemma gives us an expression for the odometer.
\begin{lemma}\label{Odmeqdim1}
	Let $\rho$ be a nonnegative bounded function on $\mathbb{Z}^2$ with finite support. Then the odometer function for the truncated $\alpha$-stable sandpile started with mass $\rho$ satisfies 
	$$u=s-\gamma,$$
	where $\gamma$ satisfies $L_{M,1}\gamma=\rho-1$ and $s$ given by \eqref{majorantdim1} is the least superharmonic majorant of $\gamma$.  
\end{lemma}
\begin{proof} 
	The odometer function is nonnegative, thus $u+\gamma\geq\gamma$. Since $\nu\leq 1$, by \eqref{L1(u)} we have $L_{M,1}(u+\gamma)\leq0$, hence, $s\leq u+\gamma$ by \eqref{majorantdim1}.\\
	For the converse inequality let $f\geq\gamma$ be any superharmonic function lying above $\gamma$, then $L_{M,1}(f-\gamma-u)\leq0$ in the domain $D=\{x\in\mathbb{Z}^2;\nu(x)=1\}$ of fully occupied sites. Outside $D$, since $u$ vanishes $f-\gamma-u\geq0$ and, hence, using a maximum principle for $L_{M,1}$, we conclude that it is nonnegative everywhere. 
\end{proof}
The expression obtained for $u$ in Lemma \ref{Odmeqdim1} is called the \textit{odometer equation}. Lemma \ref{Odmeqdim1} gives us a way to find the odometer function once we have a function $\gamma$ which satisfies $L_{M,1}\gamma=\rho-1$. This function is called the obstacle function. 

Consider the random walk $\{S_m\}_{m\geq0}$ starting at 0 whose jump probability is given by $p$; we define its Green's function $G_{M,1}$ as
\begin{equation}\label{GM1}
	G_{M,1}(x)=\lim_{N\rightarrow\infty}\sum_{m=0}^N [p^m(0)-p^m(x)],
\end{equation}
where $p^m(x)=\mathbb P[S_m=x]$.

Sometimes in literature the function $G_{M,1}$ defined in \eqref{GM1} is called the potential kernel; see, for instance, \cite{LawlerIntersections} or \cite{LawlerLimic}. Using the Markov property we can check that 
$$L_{M,1}G_{M,1}(x)=0 \text{~~~for~~~} x\neq 0 \text{~~~and~~~} L_{M,1}G_{M,1}(0)=-1.$$

Also notice that, by symmetry of the jump probability $p$, we have 
\begin{equation}\label{Lquadra}
	L_{M,1}|x|^2= \!\!\!\!\!\!\sum_{y\in B_M^{::}\cup\partial B_M^{::}}  \!\!\!\!\!\! p(y)(2(x\cdot y)+|y|^2)= \!\!\!\!\!\!\sum_{y\in B_M^{::}\cup\partial B_M^{::}}  \!\!\!\!\!\! p(y)|y|^2.
\end{equation}
Notice that the term in the right hand side of the expression above represents the variance of the random walk with transition probability $p$. Then we write
\begin{equation}\label{sigma1}
	\sigma_{M,1}^2=L_{M,1}|x|^2.
\end{equation}
Then the obstacle function $\gamma$ can, for instance, be equal to
\begin{equation}\label{obstacledim1}
	\gamma(x)=-\frac{|x|^2}{\sigma_{M,1}^2}-\sum_{y\in\mathbb{Z}^2}G_{M,1}(x-y)\rho(y) \text{~for all~}x\in\mathbb{Z}^2.
\end{equation}

\subsection{Scaling procedure}

We will introduce a  notation for the transition from the Euclidean space to the rescaled lattice $\frac{1}{n}\mathbb{Z}^2$. 

For $x\in\mathbb{R}^2$ we write $x^{n,::}$ for the nearest point in $\frac{1}{n}\mathbb{Z}^2$ breaking ties to the right, that is $x^{n,::}=(x-\frac{1}{2n},x+\frac{1}{2n}]^2\cap\frac{1}{n}\mathbb{Z}^2$.

For $x\in\frac{1}{n}\mathbb{Z}^2$, we write $x^{n,\square}=x+[-\frac{1}{2n},\frac{1}{2n}]^2$.

For a function $f$ in $\mathbb{R}^2$, we write $f^{n,::}$ for its restriction to the lattice, that is, $f^{n,::}=f|_{\frac{1}{n}\mathbb{Z}^2}$.

For a function $f$ in $\frac{1}{n}\mathbb{Z}^2$, we write $f^{n,\square}$ for its extension to the Euclidean space, defined as $f^{n,\square}(x)=f(x^{n,::})$.

For a domain $A\subset\mathbb{R}^2$, we write $A^{n,::}=A\cap\frac{1}{n}\mathbb{Z}^2$. 

If $A$ is a domain in $\frac{1}{n}\mathbb{Z}^2$, we write $A^{n,\square}=A+[-\frac{1}{2n},\frac{1}{2n}]^2$.

For a domain $A$ in $\mathbb{R}^2$, we define its discrete boundary with respect to the rescaled lattice $\frac{1}{n}\mathbb{Z}^2$ as
$$\partial A^{n,::}=\{x\in\tfrac{1}{n}\mathbb{Z}^2;x^{n,\square} ~/\!\!\!\!\!\!\subset A\text{~and~}x^{n,\square}\cap A\neq\emptyset\}.$$

In the above notations, the index $n$ will be omitted whenever it is clear from the context that we are making reference to the rescaled lattice $\frac{1}{n}\mathbb{Z}^2$.

Recall the notation $A_{\frac{r}{n},M}=B_M\setminus B_{\frac{r}{n}}$. We run our sandpile model in $\frac{1}{n}\mathbb{Z}^2$ distributing the excess mass at each toppling, proportionally to the step distribution of the random walk with jump probability $p_n:\frac{1}{n}\mathbb Z^2\to [0,1]$, given by
\begin{equation}\label{probdimn}
	p_n(y)=c_n\left(k_n\textbf{1}_{\{y\in \tfrac{1}{n}\mathbb{Z}^2,|y|=\tfrac{1}{n}\}}+\frac{F_n(y)}{n^{2+\alpha}|y|^{2+\alpha}}\right)\text{~~for~~}y\in\frac{1}{n}\mathbb Z^2\setminus\{0\}
\end{equation}
and $p_n(0)=0$, where the function $F_n$ is defined as
\begin{equation}\label{FnDef}
	F_n(y)=n^2\times\mu(y^\square\cap A_{\frac{r}{n},M}) \text{~for~}y\in\frac{1}{n}\mathbb{Z}^2.
\end{equation}
Recall that $\mu$ denotes the Lebesgue measure in $\mathbb{R}^2$. Notice that $F_n$ is equal to $1$ in $A_{\frac{r}{n},M}^{::}\setminus\partial A_{\frac{r}{n},M}^{::}$, and equal to $0$ outside $A_{\frac{r}{n},M}^{::}\cup\partial A_{\frac{r}{n},M}^{::}$.

The constant $c_n$ is the normalizing constant, and the constant $k_n$ is equal to
\begin{equation}\label{kn}
	k_n=\frac{n^{2-\alpha}}{4}\left(\int_{B_M}\frac{1}{|y|^\alpha}dy-\frac{1}{n^2}\sum_{y\in B_M^{::}\cup\partial B_M^{::}}\frac{F_n(y)}{|y|^\alpha}\right).
\end{equation}
As we mentioned before, we choose $r$ and $M$ in order to ensure that $k_1$ defined in \eqref{kdim1} is positive. The existence of such parameters is the subject of Lemma \ref{knpositive}. Moreover, as a consequence of the proof of Lemma \ref{knpositive} we also have that $k_n$ is also positive for the same choice of the parameters $r$ and $M$. In Lemmas \ref{knpositive} and \ref{cn-to-c} we prove that $k_n$ and $c_n$ are converging to positive constants. Note that the truncation remains macroscopically constant.

For a function $f$ defined in $\frac{1}{n}\mathbb Z^2$, define the discrete, truncated fractional Laplacian as 
\begin{equation}\label{lapldimn}
	L_{M,n}f(x)=n^\alpha\!\!\!\!\sum_{y\in B_M^{n,::}\cup\partial B_M^{n,::}}p_n(y)(f(x+y)-f(x)).
\end{equation}
Note that we can also write $L_{M,n}$ as 
\begin{equation}\label{Ln=Delta+Lhat}
	L_{M,n}=\frac{4c_nk_n\Delta_n}{n^{2-\alpha}}+c_n\widehat{L}_{M,n},
\end{equation}
where 
\begin{eqnarray*}
	\Delta_nf(x)&=&\frac{n^2}{4}\sum_{y\in\frac{1}{n}\bb{Z}^2, |y|=\frac{1}{n}}f(x+y)-f(x)\\
	\text{and~~}\widehat{L}_{M,n}f(x)&=&\frac{1}{n^2}\sum_{y\in B_M^{n,::}\cup\partial B_M^{n,::}}\frac{F_n(y)(f(x+y)-f(x))}{|y|^{2+\alpha}}
\end{eqnarray*}
for a function $f$ defined on $\frac{1}{n}\bb{Z}^2$. In Lemma \ref{LntoL} we prove the convergence of the discrete operator $L_{M,n}$.

Suppose we start the sandpile in $\frac{1}{n}\mathbb{Z}^2$ with initial distribution given by the function $\rho_n=\rho^{n,::}$. Similarly as we did in \eqref{obstacledim1}, we define the obstacle $\gamma_n$ as a function satisfying 
\begin{equation}\label{obstaclecondition}
	L_{M,n}\gamma_n(x)=\rho_n(x)-1.
\end{equation}
We take as the obstacle function any function that satisfies \eqref{obstaclecondition}. The obstacle function can be defined, for instance, as
\begin{equation}\label{obst-n-previa}
	\gamma_n(x)=-\frac{|x|^2}{\sigma_{M,n}^2}-\frac{1}{n^2}\sum_{y\in\frac{1}{n}\mathbb{Z}^2}G_{M,n}(x-y)\rho_n(y),
\end{equation}
where, similarly as we pointed in \eqref{Lquadra} and \eqref{sigma1}, we have that $\sigma_{M,n}^2=L_{M,n}|x|^2$ is the variance of $p_n$, and $G_{M,n}$ is a function which satisfies $L_{M,n}G_{M,n}=-n^2\delta_0$. We call $G_{M,n}$ the discrete Green's function.

In Section \ref{obstacle-function} we discuss how we define the Green's function $G_{M,n}$ and give an integral expression for it, that we anticipate here without proofs. All the details can be found in Section \ref{obstacle-function}. For a fixed $x_0\in\mathbb{R}^2$ with $|x_0|=1$, define
\begin{equation}\label{greenformuladimn}
	G_{M,n}(x)=\frac{1}{(2\pi)^2}\int_{[-n\pi,n\pi]^2}\frac{\cos(\theta\cdot x)-\cos(\theta\cdot x_0)}{\psi_{M,n}(\theta)}d\theta,
\end{equation}
where 
\begin{equation}\label{psiMn}
	\psi_{M,n}(\theta)=n^\alpha\!\!\!\!\!\!\sum_{y\in B_M^{::}\cup\partial B_M^{::}} \!\!\!\!\!\!p_n(y)(1-\cos(\theta\cdot y)).
\end{equation}
The least superharmonic majorant for the obstacle $\gamma_n$ is a function $s_n$ in $\frac{1}{n}\mathbb{Z}^2$ defined as
\begin{equation}\label{majorantdimn}
	s_n(x)=\inf{\{f(x); L_{M,n}f\leq0 \text{~~on~~}\tfrac{1}{n}\mathbb{Z}^2 \text{~~and~~} f\geq\gamma_n\}}.
\end{equation}
By analogy with \eqref{odometerdim1} we define the odometer function starting with mass $\rho_n$ in $\frac{1}{n}\mathbb{Z}^2$ as 
\begin{equation}\label{odometerdimn}
	u_n(x):=\frac{1}{n^2}\times\text{~total amount of mass emitted from the site $x$}.
\end{equation}
It is easy to see that the odometer equation, stated in \ref{Odmeqdim1}, also holds, that is,
\begin{equation}\label{odmeqdimn}
	u_n=s_n-\gamma_n,
\end{equation}
and, in the same way as \eqref{L1(u)}, we have that the final distribution of mass $\nu_n$ starting the sandpile with initial distribution of mass $\rho_n$ is given by
\begin{equation}\label{Ln(u)}
	\nu_n=\rho_n+L_{M,n}u_n.
\end{equation}

\subsection{Main results}
\label{sec:main}

The odometer equation \eqref{odmeqdimn} allows us to formulate our problem in a way that translates naturally to the continuum. Let $c$  be defined as $c:=\lim_{n\rightarrow\infty} c_n$, where $c_n$ is the normalizing constant in \eqref{probdimn}; in section \ref{analytical-estimates} we will prove that this limit exists and the convergence is of the order of $n^{-\alpha}$.

The continuous counterpart of the discrete, truncated fractional Laplacian  $L_{M,n}$ defined in \eqref{lapldimn} is the truncated fractional Laplacian whose integral form is given by
\begin{equation}\label{truncfractlapl}
	L_Mf(x)=\frac{c}{2}\int_{B_M}\frac{f(x+y)+f(x-y)-2f(x)}{|y|^{2+\alpha}}dy,
\end{equation}
defined for a twice differentiable, bounded function $f$.

In Proposition \ref{LntoL}, it becomes clear why $L_M$ is the natural extension of $L_{M,n}$ to the continuum. 

Let us fix a function $\rho$ in $\mathbb{R}^2$ with compact support, which represents the initial mass density. We define the continuous obstacle function as
\begin{equation}\label{obstaclecontinuous}
	\gamma(x)=-\frac{|x|^2}{\sigma_M^2}-\int_{\mathbb{R}^2} G_M(x-y)\rho(y)dy,
\end{equation}
where $\sigma_M^2=c\int_{B_M}\frac{1}{|y|^\alpha}dy$, and $G_M$ is the Green's function of the truncated fractional Laplacian defined as
\begin{equation}\label{greenformulacont}
	G_M(x)=\frac{1}{(2\pi)^2}\int_{\mathbb{R}^2}\frac{\cos(\theta\cdot x)-\cos(\theta\cdot x_0)}{\psi_M(\theta)}d\theta,
\end{equation}
where 
\begin{equation}\label{psiM}
	\psi_M(\theta)=c\int_{B_M}\frac{1-\cos(\theta\cdot y)}{|y|^{2+\alpha}}dy
\end{equation}
and $x_0\in\mathbb{R}^2$ is arbitrary but fixed.

The idea behind the definition of the function $G_M$ as the Green's function of the truncated fractional Laplacian $L_M$ is that $G_M$ is, in some sense, the inverse of the operator $L_M$. Such a property will be discussed in Lemma \ref{inversion-smooth-lemma}. Also, as we will discuss later in \ref{estimateGM}, the Green's function $G_M$ is an interpolation between the Green's function of the fractional Laplacian and the Green's function of the ordinary Laplacian.

We say that a function $f$ on $C^2$ is \emph{superharmonic} with respect to $L_M$ if $L_Mf\leq0$. Since we will often be working with functions that are not twice differentiable, it is convenient to define the truncated fractional Laplacian in a more general setting.

We extend $L_M$ by duality to a large class of distributions. A function $f\in L_{loc}^1$ is a linear functional in the dual of the space of infinitely differentiable functions with compact support $C_c^\infty$. The symmetry of the operator $L_M$ allows us to define, for $f\in L_{loc}^1$ 
$$\langle L_Mf,\phi\rangle=\langle f,L_M\phi\rangle,$$
where $\phi\in C_c^\infty$. This definition coincides with the previous one in the case $f\in C^2$.

We say a function $f\in L_{loc}^1(\mathbb{R}^2)$ is superharmonic in an open set $\Omega$ if for every nonnegative test function $\phi\in C_c^\infty(\mathbb{R}^2)$ whose support is inside $\Omega$, we have 
\begin{equation}\label{def-alter-superh}
	\langle f,L_M\phi\rangle\leq0.
\end{equation}

Throughout the text, superharmonic function will always mean superharmonic with respect to $L_M$. 

Now we consider the least superharmonic majorant
\begin{equation}\label{majorantcontinuous}
	s(x)=\inf\{f(x); f \text{~is continuous, superharmonic and~} f\geq\gamma\}.
\end{equation}
The majorant $s$ is the solution of the obstacle problem for the truncated fractional Laplacian with obstacle $\gamma$. 

The odometer function for $\rho$ is given by 
\begin{equation}\label{odometercontinuous}
	u=s-\gamma.
\end{equation}

We say that a sequence of functions $f_n\in L_{loc}^1$ converges to a distribution $T$ in the weak-$*$ topology if for any test function $\phi\in C_c^\infty$ we have
$$\lim_{n\rightarrow\infty}\langle f_n,\phi\rangle=\langle T,\phi\rangle.$$

Our main result states the uniform convergence of the odometer functions and the convergence of the final distribution of the truncated $\alpha$-stable sandpile with respect to the weak-$*$ topology.

We denote by $C_c^2(\mathbb{R}^2)$ the space of the twice differentiable functions with compact support.

\begin{theorem}\label{convergence-distribution}
	Let $\alpha\in(1,2)$. Consider $\rho\in C_c^2(\mathbb{R}^2)$ as the initial density of the truncated $\alpha$-stable sandpile. Let $\nu_n$ be the final distribution of the sandpile in $\frac{1}{n}\mathbb{Z}^2$ starting with initial distribution $\rho_n=\rho^{::}$. Then 
	$$\nu_n^\square\rightarrow \rho+L_Mu$$ 
	in the weak-$*$ topology, where $u$ is the limiting odometer function defined in \eqref{odometercontinuous}.
\end{theorem}

Theorem \ref{convergence-distribution} is a consequence of Theorem \ref{odoconvergence} bellow.

\begin{theorem}\label{odoconvergence}
	Let $\alpha\in(1,2)$. Consider $\rho\in C_c^2(\mathbb{R}^2)$ as the initial density of the truncated $\alpha$-stable sandpile. Starting with an initial distribution $\rho_n=\rho^{::}$ we consider the discrete odometer function $u_n$ defined in  \eqref{odometerdimn}. Let $u$ be the continuous odometer function defined \eqref{majorantcontinuous} in and \eqref{odometercontinuous}. Then
	$$u_n^\square\rightarrow u$$
	uniformly as $n$ goes to infinity.
\end{theorem}

As we mentioned before, we obtained the following convergence result for the Green's functions. 
\begin{theorem}\label{GMntoGM}
	Let $G_{M,n}$ and $G_M$ be, respectively, the rescaled discrete Green's function and the continuous Green's function defined in \eqref{greenformuladimn} and in \eqref{greenformulacont}. Then there exists a constant $K>0$ which only depends on $\alpha$, such that for all $x\neq0$ and all $n$, it holds 
	\begin{equation}\label{GMntoGMineq}
		\Big|G_M(x)-G_{M,n}(x)+\beta_n\log|x|\Big|\leq\frac{K}{n^\alpha}\left(1+\frac{1}{|x|^{2-\alpha}}+\frac{1}{|x|^2}\right),
	\end{equation}
	where $\beta_n$ is an explicitly computable constant (given in terms of the probability jumps) which converges to zero at rate $n^{-\alpha}$.\\
	In particular, $G_{M,n}$ converges towards $G_M$ uniformly on compacts of $\mathbb{R}^2\setminus \{0\}$. 
\end{theorem} 

\subsection{Organization of the paper}

In this work we will follow the structure 
\begin{eqnarray*}
	\text{convergence of Green's functions}&\Rightarrow&\text{convergence of obstacles}\\
	&\Rightarrow&\text{convergence of majorants}\\
	&\Rightarrow&\text{convergence of odometers}\\
	&\Rightarrow&\text{convergence of final distribution}.
\end{eqnarray*}
 In Section \ref{analytical-estimates} we put together some important analytical estimates that we will use throughout this work. In Section \ref{GREEN} we treat the convergence of the Green's functions, which is the subject of Theorem \ref{GMntoGM}. In Section \ref{convergence-obstacle} we prove the convergence of obstacles as a consequence of Theorem \ref{GMntoGM}. Although we do not use all the information that Theorem \ref{GMntoGM} gives in the proof of the convergence of obstacles, we do need such a strong result for the convergence of majorants. In Section \ref{convergence-odometers} we prove the convergence of majorants and the convergence of the odometer functions stated in Theorem \ref{odoconvergence}. A fundamental tool for Section \ref{convergence-odometers} is the continuity of the majorant \eqref{majorantcontinuous} which is the subject of Section~\ref{CONT}. In Section \ref{final-distribution} we obtain the convergence of the final distribution stated in Theorem~\ref{convergence-distribution}. Proof of various auxiliary results are in Sections \ref{proofs} and \ref{convergence-green-proofs}.

\section{Existence of the odometer function and Abelian property}

In this section we follow the approach presented in \cite{Levine2} to prove the existence of the odometer, the convergence to a final distribution, and that these functions are independent of the sequence of topplings. We only made the necessary changes to handle the long range jumps. 

We start with an initial distribution of mass $\rho$ in $\bb{Z}^2$ with finite support. At each time step we topple a full site, recall that a toppling in the site $x$ consists in leaving mass $1$ to $x$, and distribute the excess mass among the lattice proportionally to the step distribution of the random walk \eqref{probdim1}. 

In this section we prove that if every full site is toppled infinitely often, the mass distribution converges to a limiting distribution on which each site has mass less than or equal to $1$. We also prove the Abelian property which states that the final distribution does not depend on the order of the topplings. 

We define a toppling scheme $T$ as an infinite sequence of indexes in $\bb{Z}^2$ in which each full site that is initially full or becomes full after the realization of the previous toppling appears in the sequence infinitely often. We define in this way because after we topple a full site, it will have mass $1$, but eventually this site will have its mass increased when we perform a toppling in another full site within a radius $M$.

We also say that a toppling scheme is legal if it only topples full sites.

Let $\nu_k^T$ be the mass distribution after the toppling of the first $k$ sites listed in $T$, and $u_k^T(x)$ the mass emitted from $x$ up to the $k$-th toppling. 

\begin{proposition}\label{existence-odometer}
	Suppose we start with an initial configuration $\rho$ with finite total amount of mass and bounded support. Consider a legal toppling scheme $T$ for this initial configuration. Then as $k$ goes to infinity, $\nu^T_k$ and $u^T_k$ tend to limits $\nu$ and $u$. Moreover the limiting configuration $\nu$ satisfies $\nu\leq1$ in $\bb{Z}^2$ and $\nu=\rho+L_{M,1}u$, for $L_{M,1}$ defined in \eqref{lapldim1}.
\end{proposition}
\begin{proof}
	We claim that there exists a bounded set $K$ which contains all the possible sites with positive mass throughout the realization of the topplings. To see that, denote by $S$ the support of the initial distribution $\rho$ and let $m$ be the total amount of mass. If a lattice site $y$ outside $S$ becomes full at a certain instant of the toppling scheme $T$, then there exists a sequence $y_1,y_2,\cdots,y_\ell$ of full sites such that $y_1=y$, $y_\ell\in S$, and $y_{i+1}\in \{y_i\}+B_{M+1}^{1,::}$. Note that, since the mass is preserved, $\ell\leq m$. Then all possible full sites are inside the set $\widetilde{K}:=S+B_{m(M+1)}$. This allows us to conclude that the final set of sites with positive mass is contained in the set $K:=\widetilde{K}+B_{M+1}$. 
	
	Let us consider the function
	$$W_k=\sum_{x\in\bb{Z}^2}\nu^T_k(x)\frac{|x|^2}{\sigma_{M,1}^2},$$
	where $\sigma_{M,1}^2$ was defined in \eqref{sigma1}.
	
	Note that $W_k$ is uniformly bounded.
	
	Assume that $x$ is the $k$-th site to be toppled, then the mass at point $x$ is modified by $\beta_k(x)=\nu^T_{k-1}(x)-\nu^T_k(x)$, and this amount will be transfered to the lattice according to the toppling rule. Hence,
	\begin{eqnarray*}
		W_k(x)-W_{k-1}(x)&=&\sum_{y\in\bb{Z}^2}(\nu^T_k(y)-\nu^T_{k-1}(y))\frac{|y|^2}{\sigma_{M,1}^2}\\
		&=&-\beta_k(x)\frac{|x|^2}{\sigma_{M,1}^2}+\sum_{y\neq x}\beta_k(x)p(y-x)\frac{|y|^2}{\sigma_{M,1}^2}\\
		&=&\beta_k(x)L_{M,1}\frac{|x|^2}{\sigma_{M,1}^2}\\
		&=&\beta_k(x),
	\end{eqnarray*}
	where in the last line we used \eqref{Lquadra} and \eqref{sigma1}. Since $u^T_k$ is the sum up to $k$ of all the relevant $\beta_i(x)$, we get
	$$W_k=W_0+\sum_{x\in\bb{Z}^2}u^T_k(x).$$
	Note that, by definition, the sequence $(u^T_k)_k$ is increasing. Now we also have that the if sequence is bounded, then it converges to a certain function $u^T(x)$.
	
	As we observed in \ref{L1(u)}, the relation $\nu_k^T(x)=\rho(x)+L_{M,1} u_k^T(x)$ holds for all finite time $k$, so the convergence of $\nu_k$ is a consequence of the convergence of $u_k$. Moreover, its limit $\nu^T$ satisfies $\nu^T=\rho+L_{M,1} u^T$.
	
	Finally a point $x$ is either never toppled, in which case we have $\nu_k^T(x)\leq 1$ for all $k$, or it is toppled infinitely often, and then when a toppling occurs at time $k$ at the point $x$ we have $\nu_k^T(x)\leq 1$. In both cases we conclude that the limit $\nu^T$ satisfies $\nu^T(x)\leq 1$.
\end{proof}

The next proposition states the Abelian property for the truncated $\alpha$-stable sandpile. 

\begin{proposition}\label{abelianprop}
	Consider two legal toppling schemes $T_1$ and $T_2$ for an initial configuration $\rho$. For $i=1,2$, let $\nu^{T_i}$ and $u^{T_i}$ be, respectively, the final distribution and the odometer function for the toppling scheme $T_i$.\\
	Then $\nu^{T_1}=\nu^{T_2}$ and $u^{T_1}=u^{T_2}$.
\end{proposition}

\begin{proof}
	Assume that at time $k$ we topple the site $x_k$ in the toppling scheme $T_1$. We will prove by induction that $u^{T_2}(x_k)\geq u_k^{T_1}(x_k)$.
	
	This property is trivially true for $k=1$. Suppose that this holds for all $i<k$. For $x\neq x_k$ we have two possibilities: either $x$ is not toppled before time $k$ in the scheme $T_1$, in which case $u_k^{T_1}(x)=0$, or $x$ is toppled before time $k$. In this case we consider the last index $i$ in which the site $x$ is toppled. Then $u^{T_2}(x)\geq u_i^{T_1}(x)=u_k^{T_1}(x)$. In both cases
	\begin{equation}\label{abelian1}
		u^{T_2}(x)\geq u_k^{T_1}(x).
	\end{equation} 
	Since $T_1$ is a legal toppling scheme,
	\begin{equation}\label{abelian2}
		\nu^{T_2}(x_k)\leq 1\leq \nu_k^{T_1}(x_k).
	\end{equation}
	The initial configuration is the same for both toppling schemes, so \eqref{abelian2} says
	$$L_{M,1} u^{T_2}(x_k)\leq L_{M,1} u_k^{T_1}(x_k).$$
	That is,
	$$u^{T_2}(x_k)-u_k^{T_1}(x_k)\geq \sum_{x\neq x_k}p(x-x_k)(u^{T_2}(x)-u_k^{T_1}(x)).$$
	
	Each term in the sum above is nonnegative by \eqref{abelian1}; then we conclude that $u^{T_2}(x_k)\geq u_k^{T_1}(x_k)$. Since the site $x_k$ is toppled infinitely many times, it follows that $u^{T_2}\geq u^{T_1}$. The same argument shows the reversed inequality, which concludes the proof.
\end{proof}

From now on, we can omit the index $T$ on $u^T$ and $\nu^T$ since these functions do not depend on the sequence of topplings.

\section{The obstacle function}\label{obstacle-function}

The odometer equation stated in \eqref{odmeqdimn} gives us a way to find the odometer function for the truncated $\alpha$-stable sandpile, provided that we have a function $\gamma_n$, that we call the obstacle function, which satisfies condition \eqref{obstaclecondition}. It is important to notice that the obstacle function is not unique, only its Laplacian needs to be fixed. However, by \eqref{odmeqdimn}, we check that the odometer function $u_n$ is unique and only depends on the initial configuration $\rho_n$.

This section is devoted to the construction of our obstacle function $\gamma_n$. We follow the approach of \cite{Levine2} making our construction using the Green's functions. However, in order to have convergence of the discrete Green's functions, we need to add a special sequence of constants. This strategy is valid because we are only interested in preserving the value of $L_{M,n}\gamma_n$.

We want a function $G_{M,n}$ which satisfies $L_{M,n}G_{M,n}(x)=0$ for $x\neq0$ and $L_{M,n}G_{M,n}(0)=-n^2$. We know that the Green's function of the random walk with transition probability $p_n$ satisfies this property. See \eqref{GM1} for the definition of the Green's function of the random walk with jump probability $p$.

Using Fourier's transform, as in \cite{LawlerLimic}, we write this function as 
$$\widehat{G}_{M,n}(x)=\frac{1}{(2\pi)^2}\int_{[-n\pi,n\pi]^2}\frac{\cos(\theta \cdot x)-1}{\psi_{M,n}(\theta)}d\theta,$$
where 
$$\psi_{M,n}(\theta)=n^\alpha\!\!\!\!\!\!\sum_{y\in B_M^{::}\cup\partial B_M^{::}} \!\!\!\!\!\!p_n(y)(1-\cos(\theta\cdot y)),$$
as defined in \eqref{psiMn}.

In Section \ref{GREEN}, especially in Lemmas \ref{psiM-pequeno-estimate} and \ref{psiM-grande-estimate}, we will have a better understanding of the behavior of the function $\psi_{M,n}$, so we have that $\widehat{G}_{M,n}$ diverges as we send $n$ to infinity. To overcome this difficulty we add a divergent sequence of constants in a way that allows us to obtain convergence. So we fix some $x_0$ in $R^2$ such that $|x_0|=1$ and define

$$G_{M,n}(x)=\frac{1}{(2\pi)^2}\int_{[-n\pi,n\pi]^2}\frac{\cos(\theta \cdot x)-\cos(\theta \cdot x_0)}{\psi_{M,n}(\theta)}d\theta.$$
Recall that we previously introduced this function in \eqref{greenformuladimn}.

Define the convolution
$$G_{M,n}*\rho_n(x)=\frac{1}{n^2}\sum_{y\in\frac{1}{n}\bb{Z}^2}G_{M,n}(x-y)\rho_n(y).$$
The operator $L_{M,n}$ commutes with the sum, therefore,
$$L_{M,n}(G_{M,n}*\rho_n)(x)=(L_{M,n}G_{M,n})*\rho_n(x)=-\rho_n(x).$$
Now we are ready to define the obstacle function $\gamma_n$ on $\frac{1}{n}\bb{Z}^2$ as
\begin{equation}\label{obstacledimn}
	\gamma_n(x)=-\frac{|x|^2}{\sigma_{M,n}^2}-G_{M,n}*\rho_n(x),
\end{equation}
where $\sigma_{M,n}^2=L_{M,n}|x|^2$, as noticed in \eqref{Lquadra} and \eqref{sigma1}.

\section{Analytical estimates}\label{analytical-estimates}

In this section we put together some of the analytical estimates that we use throughout this paper. We also state some qualitative properties of the Green's function, the obstacle function, and the odometer function. The proofs of these propositions can be found in Section \ref{proofs}. 

Recall from \eqref{kdim1} the definition of $k_1$. The next lemma states that we can choose the parameters in the definition of $k_1$ in order to ensure that $k_1$ is positive and, consequently, the probability jump $p$ is well defined. 
\begin{lemma}\label{knpositive}
	There exists a positive constant $C$ such that for all $r>C$ and $M>r$ we have that $k_1$ is positive.
\end{lemma}

The next lemma tells us a simple fact concerning the normalizing constants $c_n$ of the jump probability $p_n$ (recall from \eqref{probdimn}), and the constant $c=\lim_{n\rightarrow\infty}c_n$.
\begin{lemma}\label{cn-to-c}
	There exists a positive constant $C$ which depends on $\alpha$, $r$, and $M$ such that 
	$$|c_n-c|\leq\frac{C}{n^\alpha}\text{~for all~}n.$$
\end{lemma}

In the next lemma we show that, for a sufficiently smooth function $f$, defined in $\bb R^2$, its discrete fractional Laplacian $L_{M,n}$, defined in \eqref{lapldimn}, approximates the fractional truncated Laplacian $L_M$ defined in \eqref{truncfractlapl}. 
\begin{lemma}\label{LntoL}
	If $f:\bb R^2\to\bb R$ is twice differentiable in $\bb R^2$ and $K\subset\bb{R}^2$ is a compact set, then
	$$|L_Mf(x)-L_{M,n}f(x)|\leq\frac{C}{n^{2-\alpha}}$$
	for $x\in K$, where the constant $C$ depends on $\alpha$ and on the first and second partial derivatives of the function $f$ in $K+B_M$.
\end{lemma}

As we mentioned before, the Green's function $G_M$ (recall the definition from \eqref{greenformulacont}), is in some sense the inverse of the truncated fractional Laplacian $L_M$. This is the subject of the next lemma.
\begin{lemma}\label{inversion-smooth-lemma}
	Let $\phi:\bb{R}^2\rightarrow\bb{R}$ be an infinitely differentiable function with compact support. Then 
	$$L_M(G_M*\phi)=-\phi.$$
\end{lemma}

The next two propositions give us a better understanding of the behavior of the Green's function $G_M$ defined in \eqref{greenformulacont}. 

A change of variables shows that 
\begin{equation}\label{changegreen}
	G_M(x)=\frac{1}{M^{2-\alpha}}\Big(G_1(\tfrac{x}{M})-G_1(\tfrac{x_0}{M})\Big).
\end{equation}

We have the following estimates for the function $G_1$.
\begin{proposition}\label{estimateG1}
	There exists a bounded function $g$ in $B_1$, and a bounded function $h$ in $B_1^c$, so that
	\begin{equation}\label{estimateG1-1}
		G_1(x)=\frac{\alpha^2}{c(2\pi)^2}\frac{1}{|x|^{2-\alpha}}+g(x)\text{~for all~}|x|<1
	\end{equation} 
	and
	\begin{equation}\label{estimateG1-2}
		G_1(x)=-\frac{2}{\pi\sigma_1^2}\log|x|+h_1(x)+\frac{h(x)}{|x|^{2-\alpha}}\text{~for all~}|x|\geq 1,
	\end{equation}
	where $h_1$ is the bounded function in $B_1^c$, given by $h_1(x)=\frac{1}{\pi^2\sigma_1^2}\int_{B_{|x|}\setminus B_1}\frac{\cos(\theta\cdot \omega)}{|\theta|^2}d\theta$, where $\omega$ is any unitary vector. 
\end{proposition}

We have the following estimates for $G_M$.
\begin{proposition}\label{estimateGM}
	There exists a positive constant $C$ independent of $x$ and $M$, such that
	$$\Big|G_M(x)-\frac{\alpha^2}{c(2\pi)^2}\frac{1}{|x|^{2-\alpha}}+\frac{\alpha^2}{c(2\pi)^2}\Big|<\frac{C}{M^{2-\alpha}}\text{~for all~}x\in B_M$$
	and
	$$\Big|G_M(x)+\frac{2}{\pi\sigma_M^2}\log|x|-h_M(x)-\delta_M\Big|<\frac{C}{|x|^{2-\alpha}}\text{~for all}~x\in B_M^c,$$
	where $\delta_M$ is an explicitly computable constant, converging to $-\frac{\alpha^2}{(2\pi)^2c}$ as $M$ goes to infinity, and $h_M$ is a bounded function in $B_M^c$, given by the formula $h_M(x)=\frac{1}{\pi^2\sigma_M^2}\int_{B_{\frac{|x|}{M}}\setminus B_1}\frac{\cos(\theta\cdot\omega)}{|\theta|^2}d\theta$, where $\omega$ is any unitary vector.
\end{proposition}

The function $G(x):=\frac{\alpha^2}{c(2\pi)^2}\frac{1}{|x|^{2-\alpha}}$ is the Green's function of the fractional Laplacian whose integral form is given by (see \cite{SilvestreTesis}) 
\begin{equation}\label{fractlapl}
	Lf(x)=\int_{\bb{R}^2}\frac{f(x+y)+f(x-y)-2f(x)}{|y|^{2+\alpha}}dy
\end{equation} 
for a twice differentiable, bounded function $f$.

Also note that $\log|x|$ is the Green's function of the classical Laplacian (see \cite{LawlerIntersections}); then Proposition \ref{estimateGM} states that the Green's function $G_M$ is an interpolation between the Green's function of the classical Laplacian and the Green's function of the fractional Laplacian.

The corollary bellow is a direct consequence of Proposition \ref{estimateGM}.
\begin{corollary}\label{GMtoG}
	The function $G_M(x)$ converges to $G(x)-\frac{\alpha^2}{(2\pi)^2c}$ uniformly on compacts sets of $\bb{R}^2$, as $M$ goes to infinity.
\end{corollary}

In some situations we will need to integrate the first partial derivatives of the Green's function $G_M$ over finite sets. In the following proposition we justify such operations. 
\begin{proposition}\label{GM-derivada}
	The first partial derivatives of the Green's function $G_M$ are locally integrable.
\end{proposition}

Later on we will use that the odometer function $u$ defined in \eqref{odometercontinuous} has compact support. The next lemma will be used to prove this fact. 

\begin{lemma}\label{obstacle-cont-concave}
	There exists a ball $\Omega$ such that the continuous obstacle function $\gamma$ defined in \eqref{obstaclecontinuous} is concave outside $\Omega$.
\end{lemma}

\begin{proposition}\label{odom-supp-cont}
	The continuous odometer function $u$ defined in \eqref{odometercontinuous} has compact support. 
\end{proposition}

\section{Continuity of the majorant}\label{CONT}

In this section we prove the continuity of the majorant $s$ defined in \eqref{majorantcontinuous}.

The continuity of the majorant, in the previous related works (see, for example,~\cite{Levine2}), is a simple consequence of the mean value inequality for the classical Laplacian operator. We can translate the same technique to our context and, for that, in this section we present a mean value inequality for the truncated fractional Laplacian. We follow the approach of Chapters 2 and 3 of Silvestre~\cite{SilvestreTesis}. However, in our case, extra care must be taken since we are dealing with a truncation parameter that changes when we do simple operations like change of variables on integration. 

Recall from \eqref{truncfractlapl} and \eqref{greenformulacont} the definition of the truncated fractional Laplacian $L_M$ and its respective Green's function $G_M$. For $\lambda>0$, consider the Green's function $G_{\frac{M}{\lambda}}$.

We will construct a family of approximations of the identity. For every $\lambda>0$ we stick a paraboloid from below to cut out the singularity of the function $G_{\frac{M}{\lambda}}$ at $x=0$. In this way we obtain a twice differentiable function that we call $\Gamma_\lambda(x)$. The parameters of the paraboloid can be chosen in a way such that the function $\Gamma_\lambda(x)$ coincides with $G_{\frac{M}{\lambda}}(x)$ when $x$ is outside a ball of radius $1$, and inside the ball $B_1$ the function $\Gamma_\lambda(x)$ coincides with the paraboloid. 

Next we rescale $\Gamma_\lambda$ in the following way:
\begin{equation}\label{Gammatilde_lambda}
	\widetilde{\Gamma}_\lambda(x)=\tfrac{1}{\lambda^{2-\alpha}}\Gamma_\lambda\left(\tfrac{x}{\lambda}\right).
\end{equation}
A change of variables shows that
\begin{equation}\label{changegreenlambda}
	G_{\frac{M}{\lambda}}(x)=\lambda^{2-\alpha}G_M(\lambda x)-\lambda^{2-\alpha}G_M(\lambda x_0),
\end{equation} 
so the function $\widetilde{\Gamma}_\lambda$ coincides with $G_M(x)-G_M(\lambda x_0)$ when $x$ is outside $B_\lambda$.

Let us consider then the family of functions
\begin{equation}\label{Gammahat_lambda}
	\widehat{\Gamma}_\lambda=\widetilde{\Gamma}_\lambda+G_M(\lambda x_0)
\end{equation}
and note that $\widehat{\Gamma}_\lambda$ coincides with $G_M$ outside $B_\lambda$. Moreover $\widehat{\Gamma}_{\lambda_1}\leq\widehat{\Gamma}_{\lambda_2}$ if $\lambda_1\geq\lambda_2$. 

\begin{proposition}\label{propcont-1}
	For all $\lambda>0$, $-L_M\widehat{\Gamma}_\lambda$ is a positive continuous function in $L^1$. Moreover  
	$$\int_{\bb{R}^2}-L_M\widehat{\Gamma}_\lambda(x)dx=1.$$
\end{proposition}
\begin{proof}
	We have the following rescaling property
	$$L_M\widehat{\Gamma}_\lambda(x)=L_M\widetilde{\Gamma}_\lambda(x)=\frac{1}{\lambda^2}L_{\frac{M}{\lambda}}\Gamma_\lambda\left(\frac{x}{\lambda}\right).$$
	Then a change of variables shows that it is suffices to prove that for all $\lambda>0$, we have that $-L_{\frac{M}{\lambda}}\Gamma_\lambda$ is a positive continuous function in $L^1$ and $\int_{\bb{R}^2}-L_{\frac{M}{\lambda}}\Gamma_\lambda(x)dx=1$.
	
	We first prove the $-L_{\frac{M}{\lambda}}\Gamma_\lambda$ is positive.\\
	
	If $x\notin B_1$, then $\Gamma_\lambda(x)=G_{\frac{M}{\lambda}}(x)$ and for every other $y$, $\Gamma_\lambda(y)\leq G_{\frac{M}{\lambda}}(y)$; then
	\begin{eqnarray*}
		-L_{\frac{M}{\lambda}}\Gamma_\lambda(x)&=&\frac{c}{2}\int_{B_{\frac{M}{\lambda}}}\frac{2\Gamma_\lambda(x)-\Gamma_\lambda(x-y)-\Gamma_\lambda(x+y)}{|y|^{2+\alpha}}dy\\
		&\geq& \frac{c}{2}\int_{B_{\frac{M}{\lambda}}}\frac{2G_{\frac{M}{\lambda}}(x)-G_{\frac{M}{\lambda}}(x-y)-G_{\frac{M}{\lambda}}(x+y)}{|y|^{2+\alpha}}dy\\
		&=&0
	\end{eqnarray*}
	since $G_{\frac{M}{\lambda}}$ is the Green's function of $L_{\frac{M}{\lambda}}$.\\
	
	If $x\in B_1\setminus {\{0\}}$, there exist $x_1\neq x$ and a positive $\delta$ such that $G_{\frac{M}{\lambda}}(\cdot -x_1)\!+\!\delta$ touches $\Gamma_\lambda$ from above at the point $x$. Then 
	\begin{eqnarray*}
		-L_{\frac{M}{\lambda}}\Gamma_\lambda(x)&=&\frac{c}{2}\int_{B_{\frac{M}{\lambda}}}\frac{2\Gamma_\lambda(x)-\Gamma_\lambda(x-y)-\Gamma_\lambda(x+y)}{|y|^{2+\alpha}}dy\\
		&\geq& \frac{c}{2}\int_{B_{\frac{M}{\lambda}}}\frac{2G_{\frac{M}{\lambda}}(x-x_1)-G_{\frac{M}{\lambda}}(x-y-x_1)-G_{\frac{M}{\lambda}}(x+y-x_1)}{|y|^{2+\alpha}}dy\\
		&=&0
	\end{eqnarray*}
	since $G_{\frac{M}{\lambda}}$ is the Green's function of $L_{\frac{M}{\lambda}}$. In the second line above we canceled the parameter $\delta$.
	
	If $x=0$, then $\Gamma_\lambda$ attains its maximum at $x$:
	$$-L_{\frac{M}{\lambda}}\Gamma_\lambda(x)=\frac{c}{2}\int_{B_{\frac{M}{\lambda}}}\frac{2\Gamma_\lambda(x)-\Gamma_\lambda(x-y)-\Gamma_\lambda(x+y)}{|y|^{2+\alpha}}dy>0$$
	because we are integrating a positive function.
	
	To show that $\int_{\bb{R}^2}-L_{\frac{M}{\lambda}}\Gamma_\lambda(x)dx=1$ we consider a smooth function $\omega$ such that $\omega(x)\leq 1$ for every $x\in\bb{R}^2$, $\omega(x)=1$ for every $x\in B_1$ and $\supp\omega\subset B_2$. Let $\omega_R(x)=\omega(\frac{x}{R})$, then
	\begin{eqnarray*}
		\int_{\bb{R}^2}-L_{\frac{M}{\lambda}}\Gamma_\lambda(x)dx-1&=&\lim_{R\rightarrow 0}\big\langle -L_{\frac{M}{\lambda}}(\Gamma_\lambda-G_{\frac{M}{\lambda}}),\omega_R\big\rangle\\
		&=&\lim_{R\rightarrow 0}\big\langle \Gamma_\lambda-G_{\frac{M}{\lambda}},-L_{\frac{M}{\lambda}}\omega_R\big\rangle\\
		&=&0
	\end{eqnarray*}
	since $L_{\frac{M}{\lambda}}\omega_R$ goes to zero uniformly on compacts sets, and $G_{\frac{M}{\lambda}}-\Gamma_\lambda$ is an $L^1$ function with compact support.
\end{proof}

Now let us define $\eta_\lambda=-L_M\widehat{\Gamma}_\lambda$. 

\begin{proposition}\label{propcont-2}
	The family $\eta_\lambda$ is an approximation of the identity as $\lambda\rightarrow0$ in the sense that, if $f$ is a continuous function in $\bb{R}^2$, then
	$$\eta_\lambda*f\rightarrow f \text{~~when~~}\lambda\rightarrow0\text{~~uniformly on compacts}.$$
\end{proposition}
\begin{proof}
	We will prove first that the collection of functions $\{L_{\frac{M}{\lambda}}\Gamma_\lambda, \lambda>0\}$ is uniformly integrable for $\lambda$ sufficiently small. That is, for all $\epsilon>0$ there exists $R>0$ such that
	$$\int_{\bb{R}^2\setminus B_R}|L_{\frac{M}{\lambda}}\Gamma_\lambda(y)|dy<\epsilon$$
	for $\lambda$ sufficiently small.
	
	Fix $y$ outside $B_2$. Hence, since $\Gamma_\lambda(y)=G_{\frac{M}{\lambda}}(y)$, we have
	$$-L_{\frac{M}{\lambda}}\Gamma_\lambda(y)=\frac{c}{2}\int_{B_{\frac{M}{\lambda}}}\frac{2G_{\frac{M}{\lambda}}(y)-\Gamma_\lambda(y-z)-\Gamma_\lambda(y+z)}{|z|^{2+\alpha}}dz.$$
	Since for $y\neq0$, $L_{\frac{M}{\lambda}}G_{\frac{M}{\lambda}}(y)=0$, the equation above can be written as
	\begin{eqnarray*}
		-L_{\frac{M}{\lambda}}\Gamma_\lambda(y)&=&\frac{c}{2}\!\int_{B_{\frac{M}{\lambda}}}\!\!\!\!\!\!\!\frac{G_{\frac{M}{\lambda}}(y-z)-\Gamma_\lambda(y-z)}{|z|^{2+\alpha}}dz+\frac{c}{2}\!\int_{B_{\frac{M}{\lambda}}}\!\!\!\!\!\!\!\frac{G_{\frac{M}{\lambda}}(y+z)-\Gamma_\lambda(y+z)}{|z|^{2+\alpha}}dz\\
		&=:&I_1+I_2.
	\end{eqnarray*}
	Since the support of $G_{\frac{M}{\lambda}}-\Gamma_\lambda$ is contained in $\overline{B_1}$, the integral $I_1$ above can be taken over the set $B_{\frac{M}{\lambda}}\cap B(y,1)$, and the integral $I_2$ can be taken in $B_{\frac{M}{\lambda}}\cap B(-y,1)$. 
	
	Since $G_{\frac{M}{\lambda}}$ is bigger than $\Gamma_\lambda$ in $B_1$, for a positive constant $C$ independent of $y$ we write
	\begin{eqnarray*}
		I_1&\leq&C\int_{B(-y,1)}\frac{G_{\frac{M}{\lambda}}(y+z)}{|z|^{2+\alpha}}dz\\
		&\leq&\frac{C}{|y|^{2+\alpha}}\int_{B_1}G_{\frac{M}{\lambda}}(z)dz,
	\end{eqnarray*}
	where the constant $C$ has changed from one line to the other.
	
	By Corollary \ref{GMtoG}, $G_{\frac{M}{\lambda}}$ converges uniformly on $B_1$ to $G-\frac{\alpha^2}{c(2\pi)^2}$ (where $G$ is the Green's function of the fractional Laplacian) when $\lambda$ goes to $0$. 
	
	Then we can pass the limit inside the integral:
	$$\lim_{\lambda\rightarrow0}\int_{B_1}G_{\frac{M}{\lambda}}(y)dy=\int_{B_1}\left(G(y)-\frac{\alpha^2}{c(2\pi)^2}\right)dy.$$
	Since $G$ is integrable in $B_1$, for $\lambda$ sufficiently small, there exists a constant $C>0$ independent of $\lambda$ and $y$ such that
	\begin{equation}\label{Estim-I1}
		I_1\leq\frac{C}{|y|^{2+\alpha}}\text{~for~}y\in B_2^c.
	\end{equation}
	The exact same argument proves that \eqref{Estim-I1} also holds for $I_2$. Then we have
	$$-L_{\frac{M}{\lambda}}\Gamma_\lambda(y)\leq\frac{C}{|y|^{2+\alpha}},\text{~for all~}y\in B_2^c.$$
	Take $R>2$, then
	$$\int_{\bb{R}^2\setminus B_R}-L_{\frac{M}{\lambda}}\Gamma_\lambda(y)dy\leq\int_{\bb{R}^2\setminus B_R}\frac{C}{|y|^{2+\alpha}}dy=\frac{2\pi C}{\alpha R^\alpha},$$
	which proves the uniform integrability. 
	
	Now fix a compact set $K$; we will prove that $\eta_\lambda*f\rightarrow f$ as $\lambda$ goes to zero, uniformly on $K$, assuming that the function $f$ is continuous. 
	
	Recall that $\eta_\lambda(x)=-L_M\widehat{\Gamma}_\lambda(x)=-\frac{1}{\lambda^2}L_{\frac{M}{\lambda}}\Gamma_\lambda(\frac{x}{\lambda})$. We have, for $x\in K$
	\begin{eqnarray*}
		\eta_\lambda*f(x)-f(x)&=&\int_{\bb{R}^2}\eta_\lambda(x-y)(f(y)-f(x))dy\\
		&=&\int_{\bb{R}^2}-L_{\frac{M}{\lambda}}\Gamma_\lambda(y)(f(x-\lambda y)-f(x))dy\\
		&=&\int_{B_{(\frac{M}{\lambda}+1)}}-L_{\frac{M}{\lambda}}\Gamma_\lambda(y)(f(x-\lambda y)-f(x))dy,
	\end{eqnarray*}
	where in the last line we used that the support of $L_{\frac{M}{\lambda}}\Gamma_\lambda$ is contained in $\overline{B_{(\frac{M}{\lambda}+1)}}$; this can be easily seen from the fact that $\Gamma_\lambda$ coincides with $G_{\frac{M}{\lambda}}$ outside $B_1$.
	
	By the uniform integrability of $\{L_{\frac{M}{\lambda}}\Gamma_\lambda,\lambda>0\}$, we fix $\epsilon>0$ and take $R>0$ so that 
	\begin{equation}\label{prop-2-2}
		\int_{B_R^c}-L_{\frac{M}{\lambda}}\Gamma_\lambda(y)dy<\epsilon\text{~for all~}\lambda\text{~~sufficiently small}.
	\end{equation} 
	Then we write, for $x\in K$
	\begin{eqnarray*}
		\eta_\lambda*f(x)-f(x)&=&\int_{B_R\cap B_{(\frac{M}{\lambda}+1)}}-L_{\frac{M}{\lambda}}\Gamma_\lambda(y)(f(x-\lambda y)-f(x))dy\\
		&&~~~~~~~~+\int_{B_R^c\cap B_{(\frac{M}{\lambda}+1)}}-L_{\frac{M}{\lambda}}\Gamma_\lambda(y)(f(x-\lambda y)-f(x))dy\\
		&=:&J_1+J_2
	\end{eqnarray*}
	
	\emph{Estimate of} $J_1$: Since $f$ is uniformly continuous on compact sets, we choose $\lambda$ small enough, so that 
	$$|f(x-\lambda y)-f(x)|<\epsilon\text{~for all~}x\in K, y\in B_R.$$
	Then 
	$$|J_1|\leq\epsilon\int_{\bb{R}^2}-L_{\frac{M}{\lambda}}\Gamma_\lambda(y)dy=\epsilon.$$
	
	\emph{Estimate of} $J_2$: For $y\in B_{(\frac{M}{\lambda}+1)}$, we have $\lambda y\in B_{M+\lambda}\subset B_{M+1}$. Then in the integral $J_2$, we are only considering $f$ taking values in the set $K+B_{(M+1)}$. Since $f$ is continuous, in particular, is bounded in $\overline{K+B_{(M+1)}}$, let us say by $A$. 
	
	Then, by \eqref{prop-2-2}, we have
	$$|J_2|\leq 2A\int_{B_R^c}-L_{\frac{M}{\lambda}}\Gamma_\lambda(y)dy<2A\epsilon.$$
	
	Since $\epsilon$ is arbitrary, we conclude our proof.
\end{proof}

In the next proposition we prove a mean value inequality for superharmonic functions with respect to the operator $L_M$. Recall from \eqref{def-alter-superh} the definition of superharmonicity for functions that are not twice differentiable. 

\begin{proposition}\label{propcont-3}
	Let $f$ be a continuous function. Then $f$ is superharmonic in an open set $U$ if and only if 
	\begin{equation}\label{mean-value}
		f(x)\geq \eta_\lambda *f(x)
	\end{equation}
	for any $x$ in $U$ and any $\lambda$ satisfying $B(x,\lambda)\subset U$.
\end{proposition}
\begin{proof}
	Suppose $L_Mf\leq0$ in $U$.	Let $r>0$ such that, $r>\lambda_1>\lambda_2$, then $\widehat{\Gamma}_{\lambda_2}-\widehat{\Gamma}_{\lambda
		_1}$ is a nonnegative smooth function supported in $B_r$. If $L_{M}f\leq0$ in $B(x,r)$, then
	$$\big\langle L_Mf,\widehat{\Gamma}_{\lambda_2}(\cdot-x)-\widehat{\Gamma}_{\lambda
		_1}(\cdot-x)\big\rangle\leq0.$$
	Using the self-adjointness of $L_M$,
	$$\big\langle f,(-L_M)\widehat{\Gamma}_{\lambda_1}(\cdot-x)-(-L_M)\widehat{\Gamma}_{\lambda_2}(\cdot-x)\big\rangle\leq0.$$
	Therefore 
	\begin{eqnarray*}
		\big\langle f,\eta_{\lambda_1}(\cdot-x)\big\rangle&\leq&\big\langle f,\eta_{\lambda_2}(\cdot-x)\big\rangle\\
		\Rightarrow f*\eta_{\lambda_1}(x)&\leq&f*\eta_{\lambda_2}(x).
	\end{eqnarray*}
	By Proposition \ref{propcont-2}, since $f$ is continuous, we have $f*\eta_\lambda\rightarrow f$ uniformly on compact sets, as $\lambda\rightarrow0$. We take $\lambda_2\rightarrow0$ above, and this concludes our proof.
	
	The converse part follows easily. 
\end{proof}

In Proposition \ref{propcont-3} we obtained the desired mean value inequality for the truncated fractional Laplacian. Before proving the continuity of the majorant $s$ defined in \eqref{majorantcontinuous} we need a few lemmas.  

Given a bounded open set $U$, define the function
\begin{equation}\label{majorantU}
	s^U(x)=\inf\{f(x): f \text{~~is continuous, superharmonic on~}U, \text{~and~}f\geq\gamma\}.
\end{equation}
By Proposition \ref{odom-supp-cont}, there exists a ball $\Omega$ which contains the support of the odometer function $u$. We fix this ball $\Omega$ and assume that $U$ satisfies 
\begin{equation}\label{U-assumption}
	\Omega+B_{M+1}\subset U.
\end{equation} 

\begin{lemma}\label{sU-superh}
	Let $U$ be a bounded open set satisfying \eqref{U-assumption}. Then the function $s^U$ defined in \eqref{majorantU} satisfies
	$$s^U(x)\geq \eta_\lambda*s^U(x)$$
	for all $x\in\bb{R}^2$ and $\lambda$ sufficiently small.
\end{lemma}
\begin{proof}
	Note that since the infimum in \eqref{majorantU} is taken over a larger set than the infimum in the definition of $s$, we have that $s^U\leq s$.
	
	By Proposition \ref{odom-supp-cont}, $s$ coincides with $\gamma$ outside the ball $\Omega$. Then $\gamma\leq s^U\leq s$ and, hence, $s^U$ also coincides with $\gamma$ in $\Omega^c$. 
	
	Consider a continuous function $f$ bigger than or equal to $\gamma$, and superharmonic in $U$. Then $f\geq s^U$ and, moreover, by Proposition \ref{propcont-3}, for $x\in U$ and $\lambda$ sufficiently small, we have
	$$f(x)\geq \eta_\lambda*f(x)\geq \eta_\lambda*s^U(x);$$
	taking the infimum over $f$, we have 
	\begin{equation}\label{s-meanvalue}
		s^U(x)\geq \eta_\lambda*s^U(x)\text{~~for~~}x\in U.
	\end{equation}
	By our choice of $U$, the function $\gamma$ is superharmonic in $U^c$. Then, using Proposition \ref{propcont-3}, for $x\in U^c$ and $\lambda$ sufficiently small
	$$s^U(x)=\gamma(x)\geq \eta_\lambda*\gamma(x)$$
	holds.
	
	Notice that $\supp(\eta_\lambda)\subset B_{M+1}$ for $\lambda<1$. Since $U$ satisfies \eqref{U-assumption}, we can replace $\gamma$ by $s^U$ in the equation above, and obtain
	\begin{equation}\label{s=sU-1}
		s^U(x)\geq\eta_\lambda*s^U(x)\text{~for~}x\in U^c.
	\end{equation}
	By \eqref{s-meanvalue} and \eqref{s=sU-1} we conclude our proof. 
\end{proof}

The next lemma states the continuity of $s^U$.

\begin{lemma}\label{sU-cont}
	Let $U$ be a bounded open set satisfying \eqref{U-assumption}. Then the function $s^U$ defined in \eqref{majorantU} is continuous.
\end{lemma} 

\begin{proof}
	We have noticed before that $s^U$ coincides with $\gamma$ in $\Omega^c$, then it is continuous in $\Omega^c$. We will prove that it is continuous in $U$.
	
	By Lemma \ref{sU-superh}, we have 
	$$s^U\geq \eta_\lambda*s^U$$
	for $\lambda$ sufficiently small.
	
	Consider now a bounded, open set $U_1$, such that, $U_1\supset\overline{U+B_M}$.
	
	Since $\gamma$ is continuous, by Proposition~\ref{propcont-2} we have that $\eta_\lambda*\gamma\rightarrow\gamma$ uniformly in $\overline{U_1}$ as $\lambda\rightarrow0$, then for $\epsilon>0$ we take $\lambda$ small, so that 
	
	\begin{equation}\label{prop-continuidade-1}
		\eta_\lambda*s^U(x)\geq\eta_\lambda*\gamma(x)\geq\gamma(x)-\epsilon\text{~for~}x\in U_1.
	\end{equation}
	
	Note that $\eta_\lambda*s^U$ is a continuous function. By Lemma \ref{sU-superh} and Proposition \ref{propcont-3} we conclude that it is also superharmonic. 
	
	Consider the function $g=\max(\eta_\lambda*s^U+\epsilon,\gamma)$. By \eqref{prop-continuidade-1}, $g$ coincides with $\eta_\lambda*s^U+\epsilon$ in $U_1$ and, hence, it is superharmonic in $U$. Since $g$ is above $\gamma$, we have $g\geq s^U$. We have proved
	$$\eta_\lambda*s^U(x)\leq s^U(x)\leq \eta_\lambda*s^U(x)+\epsilon\text{~for~}x\in U.$$
	Thus $s^U$ is an increasing limit of continuous functions in $U$ and, hence, lower semicontinuous. Since $s^U$ is defined as an infimum of continuous functions, it is also upper semicontinuous.  
\end{proof}

By Lemmas \ref{sU-superh} and \ref{sU-cont} and Proposition \ref{propcont-3}, we have that $s^U$ is continuous and superharmonic. Since $s^U$ is above the obstacle $\gamma$, we obtain $s=s^U$. 

We have just proved the following proposition which states the continuity of $s$.

\begin{proposition}\label{continuityofS}
	The majorant $s$ defined in \eqref{majorantcontinuous} is a continuous function. Moreover, if $U$ is a bounded open set satisfying \eqref{U-assumption}, then $s=s^U$.
\end{proposition} 

\section{Convergence of Green's functions}\label{GREEN}

In this section we prove Theorem \ref{GMntoGM} which tells us about the convergence of the Green's functions. We refer to Section \ref{convergence-green-proofs} for the proofs of the auxiliary results presented here.

Recall from \eqref{greenformuladimn} and \eqref{greenformulacont} the definition of $G_{M,n}$ and $G_M$. The first step in order to prove Theorem \ref{GMntoGM} is to investigate the convergence of $\psi_{M,n}$ to $\psi_M$ defined in \eqref{psiMn} and \eqref{psiM}, respectively.

A simple change of variables shows that 
\begin{equation}\label{psi-change}
	\psi_M(\theta)=\frac{1}{M^\alpha}\psi_1(\theta M). 
\end{equation}
For $\theta\in B_{\frac{1}{M}}$, by \eqref{psi1-1} we have that
\begin{equation}\label{psiM-pequeno}
	\psi_M(\theta)=\frac{\sigma_M^2}{4}|\theta|^2+|\theta|^4g_M(\theta)
\end{equation}
for a bounded function $g_M:B_{\frac{1}{M}}\to\bb R$. 

In the same way, for $\theta\in B_{\frac{1}{M}}^c$, by \eqref{psi1-2}
\begin{equation}\label{psiM-grande}
	\psi_M(\theta)=c_\alpha|\theta|^\alpha+h_M(\theta)
\end{equation}
for a bounded function $h_M:B_{\frac{1}{M}}^c\to\bb R$. 

This suggests that we must split the integrals that define $G_M$ and $G_{M,n}$ into three different regions: $B_{\frac{1}{M}}$, $[-n\pi,n\pi]^2\setminus B_{\frac{1}{M}}$, and $\bb{R}^2\setminus [-n\pi,n\pi]^2$. Recall from \eqref{FnDef} the definition of $F_n$.

For $\theta\in [-n\pi,n\pi]^2$, denote $\widehat{\psi}_{M,n}(\theta)=\frac{1}{n^2}\sum_{y\in B_M^{::}\cup\partial B_M^{::}}\frac{F_n(y)(1-\cos(\theta\cdot y))}{|y|^{2+\alpha}}$. Then we have 

$$\psi_{M,n}(\theta)=n^\alpha c_nk_n\sum_{|y|=\frac{1}{n}}(1-\cos(\theta\cdot y))+c_n\widehat{\psi}_{M,n}(\theta).$$
Note that for $\theta\in [-n\pi,n\pi]^2$, we have $\frac{\theta}{n}$ is bounded; then we use Taylor's theorem to ensure that there exists a positive constant $C$, such that
$$\left|2\left(1-\cos\left(\frac{\theta_1}{n}\right)\right)+2\left(1-\cos\left(\frac{\theta_2}{n}\right)\right)-\frac{|\theta|^2}{n^2}\right|\leq\frac{C|\theta|^4}{n^4},$$
then
\begin{equation}\label{psiMn-hat}
	\Big|\psi_{M,n}(\theta)-c_n\widehat{\psi}_{M,n}(\theta)-\frac{c_n k_n|\theta|^2}{n^{2-\alpha}}\Big|\leq\frac{C|\theta|^4}{n^4}.
\end{equation}

The next two lemmas describe the convergence of $\psi_{M,n}$ to $\psi_M$ when $n$ goes to infinity.
\begin{lemma}\label{psiM-pequeno-estimate}
	For $\theta\in B_{\frac{1}{M}}$ we have 
	$$|\psi_{M,n}(\theta)-\frac{c_n}{c}\psi_M(\theta)|\leq\frac{C|\theta|^4}{n^2},$$
	where the constant $C$ only depends on $\alpha$.
\end{lemma}

\begin{lemma}\label{psiM-grande-estimate}
	For $\theta\in [-n\pi,n\pi]^2\setminus B_{\frac{1}{M}}$ there exists an uniformly bounded family of functions $r_n$ such that
	$$\psi_{M,n}(\theta)=\frac{c_n}{c}\psi_M(\theta)+\frac{r_n(\theta)|\theta|^{2+\alpha}}{n^2}.$$
	Moreover there exists a positive constant $C$ depending only on $\alpha$ such that sequence $r_n$ satisfies
	\begin{itemize}
		\item[(i)] $\frac{d}{d_{\theta_i}}(r_n(\theta)|\theta|^{2+\alpha})\leq C|\theta|^{1+\alpha}$ for $i=1,2,$
		\item[(ii)] $\Delta(r_n(\theta)|\theta|^{2+\alpha})\leq C|\theta|^\alpha$.
	\end{itemize} 
\end{lemma}

Let us introduce the following notations:
\begin{eqnarray*}
	\widetilde{G}_{M,n}(x)&=&\int_{B_{\frac{1}{M}}}\frac{\cos(\theta\cdot x)-1}{\psi_{M,n}(\theta)}d\theta+\int_{[-n\pi,n\pi]^2 \setminus B_{\frac{1}{M}}}\frac{\cos(\theta\cdot x)}{\psi_{M,n}(\theta)}d\theta,\\
	\widetilde{G}_M(x)&=&\int_{B_{\frac{1}{M}}}\frac{\cos(\theta\cdot x)-1}{\psi_M(\theta)}d\theta+\int_{\bb{R}^2\setminus B_{\frac{1}{M}}}\frac{\cos(\theta\cdot x)}{\psi_M(\theta)}d\theta.
\end{eqnarray*}
The next lemma states the convergence of the auxiliary functions defined above. In the following we will conclude the proof of Theorem \ref{GMntoGM} which is the subject of this section.
\begin{lemma}\label{GMntoGMtilde}
	There exists a constant $C$ which only depends on $\alpha$, such that
	$$\left|\frac{c}{c_n}\widetilde{G}_M(x)-\widetilde{G}_{M,n}(x)\right|\leq\frac{C}{n^\alpha}\left(1+\frac{1}{|x|}+\frac{1}{|x|^2}\right).$$
\end{lemma}

\begin{proof}[Proof of Theorem \ref{GMntoGM}]
	Note that
	\begin{eqnarray*}
		G_{M,n}(x)&=&\frac{1}{(2\pi)^2}(\widetilde{G}_{M,n}(x)-\widetilde{G}_{M,n}(x_0)),\\
		G_M(x)&=&\frac{1}{(2\pi)^2}(\widetilde{G}_M(x)-\widetilde{G}_M(x_0)).
	\end{eqnarray*} 
	By Lemma \ref{GMntoGMtilde} 
	$$\Big|\frac{c}{c_n}G_M(x)-G_{M,n}(x)\Big|\leq\frac{C}{n^\alpha}\Big(1+\frac{1}{|x|}+\frac{1}{|x|^2}\Big).$$
	Let us define $\beta_n=\frac{2}{\pi\sigma_M^2}\big(\frac{c}{c_n}-1\big)$, and note that $\beta_n$ converges to zero with rate $\tfrac{1}{n^\alpha}$ due to Lemma \ref{cn-to-c}. Finally, using Proposition \ref{estimateGM} we conclude the proof of Theorem~\ref{GMntoGM}.
\end{proof}

\section{Convergence of obstacles}\label{convergence-obstacle}

Fix $x\in\bb{R}^2$ and $y\in\frac{1}{n}\bb{Z}^2$ and define
\begin{equation}\label{alpha-n}
	\alpha_n^x(y):=\frac{1}{n^2}G_{M,n}(x^{::}-y)-\int_{y^\square}G_M(x-z)dz.
\end{equation} 

\begin{lemma}\label{alpha-n-lemma}
	The function $\alpha_n^x$ satisfies 
	$$|\alpha_n^x(y)|\leq C\frac{f(x^{::}-y)}{n^{1+\alpha}},$$
	where $f$ is a positive function in $L_{loc}^1$ and $C$ is a positive constant which depends on $\alpha$.
\end{lemma}

\begin{proof}
	Write
	\begin{equation}\label{alpha-n-eq}
		\alpha_n^x(y)=\frac{1}{n^2}(G_{M,n}(x^{::}-y)-G_M(x^{::}-y))+\int_{y^\square}(G_M(x^{::}-y)-G_M(x-z))dz.
	\end{equation}
	For $z\in y^\square$, by Taylor's theorem, there exists a positive constant $C$ such that
	$$|G_M(x^{::}-y)-G_M(x-z)|\leq\frac{C}{n}|\nabla G_M(x^{::}-y)|,$$
	Integrating we obtain
	\begin{equation}\label{alpha-n-eq-1}
		\int_{y^\square}|G_M(x^{::}-y)-G_M(x-z)|dz\leq\frac{C|\nabla G_M(x^{::}-y)|}{n^3}.
	\end{equation}
	By Proposition \ref{GM-derivada}, the function $|\nabla G_M(x)|$ is in $L_{loc}^1(\bb{R}^2)$. 
	
	On the other hand by Theorem \ref{GMntoGM} there exists a constant $C>0$ depending only on $\alpha$, such that
	\begin{equation}\label{alpha-n-eq-2}
		|G_M(x^{::}-y)-G_{M,n}(x^{::}-y)|\leq\frac{C}{n^\alpha}\left(\tilde{f}(x^{::}-y)+\frac{1}{|x^{::}-y|^2}\right),
	\end{equation}
	where $\tilde{f}$ is a function in $L_{loc}^1(\bb{R}^2)$.
	
	The function $\frac{1}{|x|^2}$ is not locally integrable. To overcome this difficulty, since $x^{::}$ and $y$ are both in $\frac{1}{n}\bb{Z}^2$, we write
	$$\frac{1}{n|x^{::}-y|^2}\leq\frac{1}{|x^{::}-y|}.$$
	This last inequality, together with \eqref{alpha-n-eq}, \eqref{alpha-n-eq-1}, and \eqref{alpha-n-eq-2} concludes the proof of the lemma.
\end{proof}

\begin{proposition}\label{obstacle-convergence}
	Let $\gamma_n$ and $\gamma$ be the obstacles defined in \eqref{obstacledimn} and \eqref{obstaclecontinuous}. Then as $n$ goes to infinity
	$$\gamma_n^\square\rightarrow\gamma \text{~~uniformly on compact sets of~~}\bb{R}^2.$$
\end{proposition}

\begin{proof}
	We have $\sigma_{M,n}^2\rightarrow\sigma_M^2$. We only need to prove that $(G_{M,n}*\rho_n)^\square\rightarrow G_M*\rho$ uniformly on compact sets of $\bb{R}^2$.
	
	Since $\rho_n=\rho^{::}$ we write
	\begin{eqnarray*}
		(G_{M,n}*\rho)^\square(x)-G_M*\rho(x)&=&G_{M,n}*\rho(x^{::})-G_M*\rho(x)\\
		&=&\frac{1}{n^2}\!\!\!\sum_{y\in\frac{1}{n}\bb{Z}^2}\!\!\rho(y)G_{M,n}(x^{::}-y)\!-\!\int_{\bb{R}^2}\!\!\rho(y)G_M(x-y)dy\\
		&=&\sum_{y\in\frac{1}{n}\bb{Z}^2}\rho(y)\alpha_n^x(y),
	\end{eqnarray*}
	where $\alpha_n^x(y)$ is given by \eqref{alpha-n}.
	
	Recall that $\rho$ has compact support, let us say $E$. By Lemma \ref{alpha-n-lemma}, there exists a positive locally integrable function $f$ which satisfies
	$$|G_{M,n}*\rho(x^{::})-G_M*\rho(x)|\leq\frac{C}{n^{1+\alpha}}\sum_{y\in E^{::}}f(x^{::}-y)\leq\frac{C}{n^{\alpha-1}}\int_E f(x-y)dy,$$
	where the constant $C$ has changed in the last inequality. Since $f\in L_{loc}^1$ and $\alpha>1$, we conclude the proof.
\end{proof}

\section{Convergence of odometers}\label{convergence-odometers}

In this section we prove the convergence of majorants, and obtain, as a consequence, the convergence of odometers as stated in Theorem \ref{odoconvergence}. Before that we need a few lemmas.

Recall from \eqref{obst-n-previa}, \eqref{majorantdimn}, and \eqref{odometerdimn} the definition of the discrete obstacle function $\gamma_n$, the discrete majorant $s_n$, and the discrete odometer function $u_n$, and recall the definition of their respective continuous counterparts $\gamma$, $s$, and $u$ from \eqref{obstaclecontinuous}, \eqref{majorantcontinuous}, and \eqref{odometercontinuous}, respectively.

\begin{lemma}\label{suppofodom}
	There exists a ball $\Omega$ which contains the support of the odometer functions. That is,
	$$\supp(u)\subset \Omega \text{~~and~~}\bigcup_n\supp(u_n)\subset \Omega.$$
\end{lemma}
\begin{proof}
	By Lemma \ref{obstacle-cont-concave}, there exists a ball $\Omega$ such that the obstacle function $\gamma$ is concave outside $\Omega$. In Proposition \ref{odom-supp-cont} we proved that this ball $\Omega$ also contains the support of the odometer function $u$.
	
	We can write the discrete obstacle function $\gamma_n$ as
	$$\gamma_n(x)=\frac{|x|^2}{\sigma_{M,n}^2}-(G_{M,n}-G_M^{::})*\rho_n(x)-G_M^{::}*\rho_n(x).$$
	We use Theorem \ref{GMntoGM} and repeat the argument used in the proof of Lemma \ref{obstacle-cont-concave} to conclude that there exists a ball $\Omega$ (bigger if necessary) such that the obstacles functions $\gamma_n$ are concave outside $\Omega$.
	
	Then, the concave envelope $\Gamma_n$ of the obstacle $\gamma_n$, defined as the least concave function lying above $\gamma_n$, satisfies that $\Gamma_n(x)=\gamma_n(x)$ for all $x$ outside $\Omega$ and for all $n$. The same argument used in the proof of Proposition \ref{odom-supp-cont} concludes our proof. 
\end{proof}

\begin{lemma}\label{convergencemajorantsupp}
	Let $\gamma_n$ and $s_n$ be the obstacle function and the majorant given by \eqref{obst-n-previa} and \eqref{majorantdimn}. If $\Omega\subset\bb{R}^2$ is a bounded set like the one in Lemma \ref{suppofodom}, then
	$$s_n(x)=\inf{\{f(x);L_{M,n}f\leq 0\text{~on~} \Omega^{::} \text{~and~} f\geq\gamma_n\}}.$$
\end{lemma}

\begin{proof}
	We observe first that $L_{M,n}s_n=0$ on the set $D_n=\{\nu_n=1\}$, since
	\begin{equation*}
		L_{M,n}s_n=L_{M,n}(u_n+\gamma_n)=\nu_n-\sigma_n+\sigma_n-1=\nu_n-1.
	\end{equation*}
	Let us call 
	$$s_n^\Omega(x)=\inf{\{f(x):L_{M,n}f\leq 0\text{~~on~~} \Omega^{::} \text{~~and~~} f\geq\gamma_n\}}.$$
	We want to prove that $s_n^\Omega=s_n$. Since the infimum is taken over a strictly larger set, the inequality $s_n^\Omega\leq s_n$ is trivial.
	
	For the converse inequality consider a function $f\geq\gamma_n$ such that $L_{M,n}f\leq0$ on $\Omega^{::}$. Since $L_{M,n}s_n=0$ on $D_n$, the difference $f-s_n$ satisfies $L_{M,n}(f-s_n)\leq0$ on $D_n$, then by the maximum principle $f-s_n$ attains its minimum at a point $x$ outside $D_n$, where $s_n(x)=\gamma_n(x)$. Since $f\geq\gamma_n$ we conclude $f\geq s_n$ on $\Omega^{::}$ and, hence, everywhere.
\end{proof}

Recall that Lemma \ref{inversion-smooth-lemma} states
\begin{equation}\label{inversion-smooth}
	L_M(G_M*\phi)=-\phi\text{~~for a function~}\phi\in C_c^\infty.
\end{equation}
The next lemma studies the equation above removing the smoothness assumption.

\begin{lemma}\label{inversion}
	Let $f$ be a bounded, measurable function with compact support. Then $L_M(G_M*f)=-f$ in the sense of distributions. Moreover, if $f$ is positive in an open set $U$, then $G_M*f$ is continuous and superharmonic in $U$.
\end{lemma}
\begin{proof}
	For a function $\phi\in C_c^\infty$, it holds that $L_M(G_M*\phi)=G_M*(L_M\phi)$. We have 
	$$\langle L_M(G_M*f),\phi\rangle=\langle f,L_M(G_M*\phi)\rangle=\langle f,-\phi\rangle=\langle -f,\phi\rangle.$$
	Now, by Proposition \ref{GM-derivada}, and since $f$ is bounded with compact support and the first partial derivatives of $G_M$ are locally integrable, it is easy to see that $G_M*f$ is continuous. Now, fix a positive function $\phi\in C_c^\infty$ supported in $U$. Then, we have
	$$\langle L_M(G_M*f),\phi\rangle=\langle -f,\phi\rangle\leq0,$$
	since $f$ and $\phi$ are both positive in $U$.
\end{proof}

The next proposition states the convergence of the majorants. 

\begin{proposition}\label{majorant-convergence}
	Let $s_n$ and $s$ be the majorants defined in \eqref{majorantdimn} and \eqref{majorantcontinuous}, respectively. Then
	$$s_n^\square\rightarrow s$$
	uniformly on compact sets of $\bb{R}^2$.
\end{proposition}

By Lemma \ref{suppofodom} there exists a ball $\Omega\subset\bb{R}^2$ containing $\supp(u)$ and $\supp(u_n)$ for all $n$. Outside $\Omega$, we have
$$s_n^\square=\gamma_n^\square\rightarrow\gamma=s$$
uniformly on compact sets by Proposition \ref{obstacle-convergence}. We only need to prove convergence in $\Omega$.

Let us consider four compact sets $K_1$, $K_2$, $K_3$, and $K_4$, such that $\supp(\rho)\subset K_1$, $\Omega+B_{M+1}\subset K_1$, and 
\begin{equation}\label{K-assumption}
	K_j+B_{M+1}\subset K_{j+1}\text{~for~}j=1,2,3.
\end{equation}
By our assumption about $K_1$, Lemma \ref{convergencemajorantsupp}, and Proposition \ref{continuityofS}, we have that the discrete and continuous majorants can be written as
\begin{eqnarray}
	s_n&=&\inf\{f(x); L_{M,n}f\leq 0 \text{~on~} K_1^{::} \text{~and~} f\geq\gamma_n\},\label{sn-K1}\\
	s&=&\inf\{f(x);f\text{~is continuous}, L_Mf\leq 0 \text{~on~} K_1 \text{~and~} f\geq\gamma\}\label{s-K1}.
\end{eqnarray} 
Consider the function $\omega(x)$ such that $0\leq\omega\leq1$, $\omega(x)=1$ for $x\in K_3$, and $\supp{\omega}\subset K_3+B_1$.

Define the following function on $\frac{1}{n}\bb{Z}^2$:
\begin{equation}\label{phidefinition}
	\phi_n=-L_{M,n}(s_n\omega).
\end{equation}

\begin{lemma}\label{phibounded}
	The sequence of functions $(\phi_n)_n$ is uniformly bounded. 
\end{lemma}

\begin{proof}
	Since $\supp(\omega)\subset K_3+B_1$ we have $\supp(\phi_n)\subset (K_3+B_{M+1})\subset K_4$.
	
	If $x\in K_2$, $\phi_n$ coincides with $-L_{M,n}s_n$ which is uniformly bounded by $1$, as can be easily seen as a consequence of the odometer equation.
	
	Consider then $x\in K_4\setminus K_2=:E$. In the set $E$ we can replace $s_n$ by $\gamma_n$, then in \eqref{phidefinition}, we have
	\begin{equation}\label{gamma-por-s}
		\phi_n(x)=-L_{M,n}(\gamma_n\omega)(x) \text{~~~for~}x\in E^{::}.
	\end{equation}
	We write 
	\begin{equation}\label{phibound-eq-1}
		L_{M,n}(\gamma_n\omega)(x)=L_{M,n}((\gamma_n-\gamma)\omega)(x)+L_{M,n}(\gamma\omega)(x).
	\end{equation}
	By Lemma \ref{LntoL}, the last term on the right-hand side of \eqref{phibound-eq-1} converges to $L_M(\gamma\omega)(x)$ uniformly on $E$, once we have that $\gamma\omega$ is $C^2$. We conclude that $L_{M,n}(\gamma\omega)(x)$ is uniformly bounded in $E$.\\
	We only need to check that the first term on the right-hand side of \eqref{phibound-eq-1} is also uniformly bounded on $E$.
	
	Note that
	$$L_{M,n}\left(\left(\tfrac{|\cdot|^2}{\sigma_{M,n}^2}-\tfrac{|\cdot|^2}{\sigma_M^2}\right)\omega\right)(x)=\left(\tfrac{1}{\sigma_{M,n}^2}-\tfrac{1}{\sigma_M^2}\right)L_{M,n}(|\cdot|^2\omega)(x)$$
	is uniformly bounded on $E$, also as a consequence of Lemma \ref{LntoL}. 
	
	Then, it suffices to prove that 
	$$L_{M,n}((G_M*\rho-G_{M,n}*\rho_n)\omega)$$
	is uniformly bounded on $E$.
	
	Recall from \eqref{Ln=Delta+Lhat} that we can write
	$$L_{M,n}=\frac{4c_nk_n\Delta_n}{n^{2-\alpha}}+c_n\widehat{L}_{M,n},$$
	where 
	\begin{eqnarray*}
		\Delta_nf(x)&=&\frac{n^2}{4}\sum_{y\in\frac{1}{n}\bb{Z}^2, |y|=\frac{1}{n}}f(x+y)-f(x)\\
		\text{and~~}\widehat{L}_{M,n}f(x)&=&\frac{1}{n^2}\sum_{y\in B_M^{::}\cup\partial B_M^{::}}\frac{F_n(y)(f(x+y)-f(x))}{|y|^{2+\alpha}}
	\end{eqnarray*}
	for a function $f$ defined on $\frac{1}{n}\bb{Z}^2$.
	
	Let us see that $\widehat{L}_{M,n}((G_M*\rho-G_{M,n}*\rho_n)\omega)$ is uniformly bounded on $E$. Recall that for $x\in\frac{1}{n}\bb{Z}^2$, we defined $\rho_n(x)=\rho^{::}(x)$. 
	
	Denote $E_M:=E+B_{M+1}$. For $x\in E_M^{::}$ we have
	
	\begin{eqnarray*}
		G_M*\rho(x)-G_{M,n}*\rho(x)&=&\sum_{y\in\frac{1}{n}\bb{Z}^2}\int_{y^\square}(G_M(z)\rho(x-z)-G_{M,n}(y)\rho(x-y))dz\\
		&=&I_1(x)+I_2(x),
	\end{eqnarray*}
	where
	\begin{eqnarray*}
		I_1(x)&:=&\sum_{y\in\frac{1}{n}\bb{Z}^2}\int_{y^\square}(G_M(z)\rho(x-z)-G_M(y)\rho(x-y))dz,\\
		I_2(x)&:=&\frac{1}{n^2}\sum_{y\in\frac{1}{n}\bb{Z}^2}(G_M(y)\rho(x-y)-G_{M,n}(y)\rho(x-y)).
	\end{eqnarray*}
	Since $K_1$ contains the support of the initial distribution $\rho$, the sums on the definitions of $I_1(x)$ and $I_2(x)$ are taken over a set which is safely away from the origin. Let us say $\widetilde{E}$.
	
	By symmetry, we write $I_1(x)$ as
	$$I_1(x)\!=\!\frac{1}{2}\!\!\sum_{y\in\widetilde{E}^{::}}\!\!\int_{[-\frac{1}{2n},\frac{1}{2n}]^2}\!\!\!\!\!\!\!G_M(y+z)\rho(x-y-z)+G_M(y-z)\rho(x-y+z)-2G_M(y)\rho(x-y)dz.$$
	Since the sum above has been taken over a set $\widetilde{E}$ that does not intersects a neighborhood of the origin, the function $G_M$ and, consequently, the function $G_M(\cdot)\rho(x-\cdot)$ is $C^2$, that means that using the second order Taylor's theorem we have
	$$|I_1(x)|\leq \frac{C}{n^4}\sum_{y\in\widetilde{E}^{::}}\sup_{|\beta|=2}|D^\beta(G_M(y)\rho(x-y))|\leq\frac{C}{n^2},$$
	
	where $C$ is a positive constant depending only on $\alpha$, which has changed in the last inequality.
	
	Let us now estimate $I_2$. As noticed before, the sum $I_2$ (as the sum $I_1$) is taken over a set $\widetilde{E}$ that does not intersects a neighborhood of the origin. By Theorem~\ref{GMntoGM}, we have
	$$|I_2(x)|\leq\frac{C}{n^\alpha},$$
	where  $C$ is a positive constant which only depends on $\alpha$ and on the set $\widetilde{E}$.
	We have proved that, for all $x\in E_M$
	$$|G_M*\rho(x)-G_{M,n}*\rho(x)|\leq\frac{C}{n^\alpha}.$$
	As a consequence, for $x\in E$ we have 
	$$|\widehat{L}_{M,n}((G_M*\rho-G_{M,n}*\rho)\omega)(x)|\leq\frac{C}{n^{2+\alpha}}\sum_{y\in A_{\frac{r}{n},M}^{::}\cup\partial A_{\frac{r}{n},M}^{::}}\frac{1}{|y|^{2+\alpha}}\leq C\sum_{y\in\bb{Z}^2\setminus B_r^{1,::}}\frac{1}{|y|^{2+\alpha}},$$
	which proves that $\widehat{L}_{M,n}((G_M*\rho-G_{M,n}*\rho)\omega)(x)$ is uniformly bounded in $E$.
	
	The same argument proves that $\frac{\Delta_n((G_M*\rho-G_{M,n}*\rho)\omega)}{n^{2-\alpha}}$ is uniformly bounded on $E$, which concludes the proof.
\end{proof}

\begin{lemma}\label{phi-aprox-sn}
	The function $\phi_n$ defined in \eqref{phidefinition} satisfies
	$$|s_n^{\square}-G*(\phi_n)^\square|\rightarrow 0$$
	uniformly on $K_2$.
\end{lemma}

\begin{proof}
	In the beginning of the proof of Lemma \ref{phibounded} we observed that $\supp(\phi_n)\subset K_4$. Note that
	$$G_M*(\phi_n)^\square(x)=\sum_{y\in K_4^{::}}\phi_n(y)\int_{y^\square}G_M(x-z)dz$$
	and
	$$s_n(x)=G_{M,n}*\phi_n(x) \text{~~~~for~~~~}x\in K_2^{::}.$$
	Hence for $x\in K_2$
	\begin{eqnarray*}
		s_n^\square(x)-G_M*(\phi_n)^\square(x)&=&G_{M,n}*\phi_n(x^{::})-G_M*(\phi_n)^\square(x)\\
		&=&\sum_{y\in K_4^{::}}\phi_n(y)\alpha_n^x(y)
	\end{eqnarray*}
	for $\alpha_n^x(y)$ defined in \eqref{alpha-n}.
	
	By Lemmas \ref{phibounded} and \ref{alpha-n-lemma}, there exist a positive constant $C$ independent of $n$ and $x$, and a positive function $f$ in $L_{loc}^1$ such that 
	$$|s_n^\square(x)-G_M*(\phi_n)^\square(x)|\leq\frac{C}{n^{1+\alpha}}\sum_{y\in K_4^{::}}f(y-x^{::})	\leq\frac{C}{n^{\alpha-1}}\int_{(K_2+K_4)}f(y)dy	\leq\frac{C}{n^{\alpha-1}},$$
	where the constant $C$ has changed from one inequality to the other. In the second inequality we used the approximation of the integral by Riemann's sums and the fact that $x\in K_2$ and $y\in K_4$.
\end{proof}

\begin{proof}[Proof of Proposition \ref{majorant-convergence}]
	As we pointed out before, we only need to prove that $s_n$ converges to $s$ uniformly in a ball $\Omega$ as the one considered in Lemma \ref{suppofodom}. Recall from \eqref{sn-K1} and \eqref{s-K1} the expression we are using for $s_n$ and $s$. 
	
	Recall from \eqref{K-assumption}, our assumptions about the sets $K_j$, for $j=1,2,3,4$. We will prove that $s_n$ converges to $s$ uniformly on $K_1$.
	
	We write
	
	\begin{equation}\label{mollifier}
		\tilde{s}(x)=\int_{\bb{R}^2}s(y)\lambda^{-2}\eta\left(\frac{x-y}{\lambda}\right)dy,
	\end{equation}
	where $\eta$ is the standard smooth mollifier
	\begin{equation*}
		\eta(x)=
		\begin{cases}
			\beta e^{1/(|x|^2-1)}  &  \text{~~~~if~~} |x|<1,\\
			0  &  \text{~~~~if~~} |x|\geq 1,
		\end{cases}
	\end{equation*}
	normalized so that $\int_{\bb{R}^2}\eta dx=1$ (see \cite{Evans}). Then $\tilde{s}$ is smooth and superharmonic.
	
	By Proposition \ref{continuityofS}, $s$ is continuous, and by compactness, $s$ is uniformly continuous on $K_2$, so taking $\lambda$ sufficiently small in \eqref{mollifier}, we have $|s-\tilde{s}|<\epsilon$ in $K_2$.
	
	By Lemma \ref{LntoL}, there exists a positive constant $C_1$ such that the function 
	$$q_n(x)=\tilde{s}^{::}-\frac{C_1}{n^{2-\alpha}}|x|^2$$
	satisfies $L_{M,n}q_n\leq0$ in $K_2^{::}$. By Proposition \ref{obstacle-convergence} we have $\gamma_n^\square\rightarrow\gamma$ uniformly in $K_2$. Taking $n$ large enough to ensure that $\frac{C_1}{n^{2-\alpha}}|x|^2\leq\epsilon$ in $K_2$ and $|\gamma_n-\gamma^{::}|<\epsilon$ in $K_2^{::}$, we obtain
	$$q_n>\tilde{s}^{::}-\epsilon>s^{::}-2\epsilon\geq\gamma^{::}-2\epsilon>\gamma_n-3\epsilon$$
	in $K_2^{::}$.
	
	Now, since the function $f_n=\max(q_n+3\epsilon,\gamma_n)$ coincides with $q_n+3\epsilon$ in $K_2^{::}$, then $f_n$ satisfies $L_{M,n}f_n\leq0$ in $K_1^{::}$. It follows that $f_n\geq s_n$, hence, in $K_1^{::}$, we have
	$$s_n\leq q_n+3\epsilon<\tilde{s}^{::}+3\epsilon<s^{::}+4\epsilon.$$
	By the uniform continuity of $s$ on $K_1$, taking $n$ larger if necessary, we have $|s-s^{::\square}|<\epsilon$ in $K_1$ and, hence, $s_n^\square<s+5\epsilon$ in $K_1$.
	
	For the reverse inequality, recall from \eqref{phidefinition} the definition of the function $\phi_n$. By Lemma \ref{phi-aprox-sn}, we have
	$$|s_n^\square-G_M*(\phi_n)^\square|<\epsilon$$
	in $K_2$ and, hence,
	$$G_M*(\phi_n)^\square>\gamma_n^\square-\epsilon>\gamma-2\epsilon$$ 
	in $K_2$ for $n$ sufficiently large. Since $\phi_n$ is nonnegative on $K_2$, by Lemma \ref{inversion} the function $G_M*(\phi_n)^\square$ is superharmonic on $K_2$, so the function
	$$g_n=\max(G_M*(\phi_n)^\square+2\epsilon,\gamma)$$
	coincides with $G_M*(\phi_n)^\square+2\epsilon$ in $K_2$, then $g_n$ is superharmonic on $K_1$ for sufficiently large $n$. By the definition of $s$ it follows that $g_n\geq s$, hence, 
	$$s_n^\square>G_M*(\phi_n)^\square-\epsilon\geq s-3\epsilon$$
	on $K_1$ for sufficiently large $n$.
\end{proof}

Now we conclude the proof of Therem \ref{odoconvergence}.

\begin{proof}[Proof of Theorem \ref{odoconvergence}]
	Let $\Omega$ be a ball as in Lemma \ref{suppofodom}. By Propositions \ref{obstacle-convergence} and \ref{majorant-convergence} we have $\gamma_n^\square\rightarrow\gamma$ and $s_n^\square\rightarrow s$ uniformly on $\Omega$. By \eqref{odmeqdimn} we have $u_n=s_n-\gamma_n$. Since $u_n^\square=0=u$ outside $\Omega$, we conclude $u_n^\square\rightarrow s-\gamma=u$ uniformly.
\end{proof}

\section{Convergence of the final distribution}\label{final-distribution}

Now we are in position to prove Theorem \ref{convergence-distribution}, which tells us about the weak-$*$ convergence of the final distributions.

\begin{proof}[Proof of Theorem \ref{convergence-distribution}]
	Recall from \eqref{Ln(u)} that the final distribution $\nu_n$ of the truncated $\alpha$-stable divisible sandpile in $\frac{1}{n}\bb{Z}^2$ starting with initial density of mass $\rho_n=\rho^{::}$ satisfies the equation
	$$\nu_n=\rho_n+L_{M,n}u_n.$$
	Let $\phi\in C_c^\infty(\bb{R}^2)$ be a test function. We have
	\begin{equation}\label{distribution}
		\langle\nu_n^\square,\phi\rangle=\langle (L_{M,n}(u_n-G_{M,n}*\rho_n))^\square,\phi\rangle=\langle u_n^\square-(G_{M,n}*\rho_n)^\square, L_{M,n}\phi\rangle.
	\end{equation}
	
	Lemma \ref{LntoL} ensures that $L_{M,n}\phi$ converges to $L_M\phi$ uniformly. By Theorem \ref{odoconvergence}  and Proposition \ref{obstacle-convergence} we have the convergence $(u_n-G_{M,n}*\rho^{::})^\square\rightarrow u-G_M*\rho$ uniformly on compact sets of $\bb{R}^2$. Since $\phi$ has compact support, there exists a ball which contains the support of $L_{M,n}\phi$ for all $n$. From \eqref{distribution} we have
	$$\lim_{n\rightarrow\infty}\langle\nu_n^\square,\phi\rangle=\langle u-G_M*\rho,L_M\phi\rangle=\langle\rho+L_Mu,\phi\rangle.$$
	Let us denote $\nu=\rho+L_Mu$, we have just proved that $\nu_n$ converges towards $\nu$ in the weak-$*$ topology. We must show that the distribution $\nu$ is actually a function.  
	
	By Lemma \ref{suppofodom} we can also say that there exists a ball $U$ which contains the support of the final distribution $\nu_n^\square$ for all $n$. Then we have
	$$\langle\nu_n^\square,\phi\rangle=\int_U\nu_n^\square(x)\phi(x)dx.$$
	Since $0\leq\nu_n^\square\leq 1$ for all $n$, by Holder's inequality we have
	\begin{equation}\label{nu-measure}
		|\langle\nu,\phi\rangle|=\lim_{n\to\infty}|\langle\nu_n^\square,\phi\rangle|\leq\int_U|\phi(x)|dx\leq\mu(U\cap\supp(\phi))\parallel\phi\parallel_\infty,
	\end{equation}
	where $\mu$ denotes the Lebesgue measure. This shows that $\nu$ is a continuous linear functional, then by the Riesz-Markov theorem, there exists a Radon measure $\lambda$ such that
	$$\langle\nu,\phi\rangle=\int\phi(x)d\lambda(x).$$ 
	Moreover, the measure $\lambda$ is nonnegative, since the functional $\nu$ is positive.
	
	Let $A$ be a measurable set, consider a family $(A_\epsilon)_\epsilon$ of measurable sets such that $A_\epsilon\downarrow A$ when $\epsilon\downarrow0$. Now consider a family $(g_\epsilon)_\epsilon$ of infinitely differentiable functions such that $0\leq g_\epsilon\leq 1$ and $\supp(g_\epsilon)\subset A_\epsilon$. In particular, from \eqref{nu-measure}, for all $\epsilon$ we have 
	$$\lambda(A)\leq\langle\nu,g_\epsilon\rangle\leq\mu(A_\epsilon)\parallel g_\epsilon\parallel_\infty.$$
	Taking limits as $\epsilon\downarrow0$ we obtain
	$$\lambda(A)\leq\mu(A).$$
	Then $\lambda$ is absolutely continuous with respect to the Lebesgue measure, therefore, the linear functional $\nu$ is indeed a function.
\end{proof}

\section{Proofs of analytical estimates}\label{proofs}
In this section we will proof the propositions and lemmas of Section \ref{analytical-estimates}. We will often use in our estimates a universal constant $C$ which may change from line to line and depends only on the parameter $\alpha$, except when we specify another possible dependency. 

\begin{proof}[Proof of Lemma \ref{knpositive}]
	
	For $R>0$, consider the function $F_R$ in $\bb{Z}^2$, defined as
	$$\widehat{F}_R(y)=\mu([y-\tfrac{1}{2},y+\tfrac{1}{2}]^2\cap B_R).$$
	For a domain $A$ in $\bb{R}^2$, we denote $A^{n,::}$ for its restriction to $\frac{1}{n}\bb{Z}^2$, that is, $A^{n,::}=A\cap\frac{1}{n}\bb{Z}^2$, and denote the discrete boundary as
	$$\partial A^{n,::}=\{x\in\tfrac{1}{n}\bb{Z}^2;[x-\tfrac{1}{2n},x+\tfrac{1}{2n}]^2 ~/\!\!\!\!\!\subset A\text{~and~}[x-\tfrac{1}{2n},x+\tfrac{1}{2n}]^2\cap A\neq\emptyset\}.$$
	Notice that $\widehat{F}_R$ is equal to $1$ in $B_R^{1,::}\setminus\partial B_R^{1,::}$, and equal to $0$ outside $B_R^{1,::}\cup\partial B_R^{1,::}$.
	Let us denote 
	$$k(R)=\int_{B_R}\frac{1}{|y|^\alpha}dy-\sum_{y\in B_R^{1,::}\cup\partial B_R^{1,::}\setminus \{0\}}\frac{\widehat{F}(y)}{|y|^\alpha}.$$
	Consider $R'>R$. A change of variables allows us to write
	$$k(R')-k(R)=\frac{n^{2-\alpha}}{2}\sum_{y\in A_{\frac{R}{n},\frac{R'}{n}}^{n,::}\cup\partial A_{\frac{R}{n},\frac{R'}{n}}^{n,::}}\int_{[-\frac{1}{2n},\frac{1}{2n}]^2}\Big(\frac{1}{|y+z|^\alpha}+\frac{1}{|y-z|^\alpha}-\frac{2}{|y|^\alpha}\Big)dz$$
	We use the second order Taylor's theorem and obtain
	\begin{eqnarray*}
		\nonumber|k(R')-k(R)|&\leq&\frac{C}{n^{2+\alpha}}\sum_{y\in A_{\frac{R}{n},\frac{R'}{n}}^{n,::}\cup\partial A_{\frac{R}{n},\frac{R'}{n}}^{n,::}}\frac{1}{|y|^{2+\alpha}}\\
		&\leq&\frac{C}{n^\alpha}\int_{\bb R^2\setminus B_{\frac{R}{n}}}\frac{1}{|y|^{2+\alpha}}dy\\
		&\leq&\frac{C}{R^\alpha}.
	\end{eqnarray*}
	We have proved that the sequence $\{k(R)\}_R$ is Cauchy, so it does converge. We can ensure then that there exists $S>0$ such that
	\begin{equation}\label{k(R)-bounded}
		|k(R)|\leq S \text{~~for all~~} R.
	\end{equation}
	Although the sequence $k(R)$ does converge, neither the integral nor the sum on its definition converge. Then there exists $r>0$ so that
	\begin{equation}\label{sum-diverges}
		\sum_{y\in B_r^{1,::}}\frac{1}{|y|^\alpha}>S.
	\end{equation}  
	Taking $r$ as in \eqref{sum-diverges} and $M$ bigger than $r$, by \eqref{k(R)-bounded} we have that $k_1=k_1(r,M)$ as in \eqref{kdim1} is positive. 
\end{proof}

\begin{proof}[Proof of Lemma \ref{cn-to-c}]
	Denote $k:=\lim_{n\rightarrow\infty}k_n$. In particular, from the proof of Lemma \ref{knpositive}, we have
	$$|k-k_n|\leq\frac{C}{n^\alpha}$$
	for a positive constant $C$ which only depends on $\alpha$.
	
	On the other hand, we also have
	$$\sum_{y\in\bb{Z}^2}\frac{1}{|y|^{2+\alpha}}-\sum_{y\in B_{Mn}^{1,::}\cup\partial B_{Mn}^{1,::}}\frac{F_n(\tfrac{y}{n})}{|y|^{2+\alpha}}\leq\sum_{y\in\bb{Z}^2\setminus B_{Mn}}\frac{1}{|y|^{2+\alpha}}\leq\frac{C}{n^\alpha}$$
	for a constant $C$ depending only on $\alpha$. This concludes the proof.
\end{proof}

\begin{proof}[Proof of Lemma \ref{LntoL}]
	Let us write
	$$L_{M,n}f(x)=\frac{4c_nk_n\Delta_nf(x)}{n^{2-\alpha}}+\frac{c_n}{n^2}\sum_{y\in A_{\frac{r}{n},M}^{::}\cup\partial A_{\frac{r}{n},M}^{::}}\frac{F_n(y)(f(x+y)-f(x))}{|y|^{2+\alpha}},$$
	where $\Delta_nf(x)=\frac{n^2}{4}\sum_{y\in\frac{1}{n}\bb{Z}^2,|y|=\frac{1}{n}}(f(x+y)-f(x))$.
	
	Let us suppose $x\in K$ for a compact set $K$. Since $f\in C^2$, there exists a constant $C$, which depends on the first and second derivatives of $f$ in $K$, such that
	$$\Big|\frac{4c_nk_n\Delta_nf(x)}{n^{2-\alpha}}\Big|\leq\frac{C\Delta f(x)}{n^{2-\alpha}}\leq\frac{C}{n^{2-\alpha}}.$$ 
	Denote
	$$g_x(y)=\frac{f(x+y)+f(x-y)-2f(x)}{2|y|^{2+\alpha}}.$$
	Since, by Lemma \ref{cn-to-c}, $|c_n-c|\leq\frac{C}{n^\alpha}$ for $C>0$ depending only on $\alpha$, it suffices to prove that 
	\begin{equation}\label{LntoLineq}
		\left|\int_{B_M}g_x(y)dy-\frac{1}{n^2}\sum_{y\in A_{\frac{r}{n},M}^{::}\cup\partial A_{\frac{r}{n},M}^{::}}F_n(y)g_x(y)\right|\leq\frac{C}{n^{2-\alpha}}
	\end{equation}
	for a constant $C$ depending on the first and second derivatives of $f$ in $K+B_M$.
	
	In order to prove the inequality \eqref{LntoLineq} we write the left-hand side as
	$$\int_{B_{\frac{r}{n}}}g_x(y)dy+\frac{1}{2}\sum_{y\in A_{\frac{r}{n},M}^{::}\cup\partial A_{\frac{r}{n},M}^{::}}\int_{[-\frac{1}{2n},\frac{1}{2n}]^2}\tilde{g}_{x,y}(z)dz,$$
	where $\tilde{g}_{x,y}(z)=g_x(y+z)+g_x(y-z)-2g_x(y)$. For $y$ small, we have $|g_x(y)|\leq\frac{C}{|y|^\alpha}$ for $x\in K$, with $C$ depending on the second partial derivatives of $f$ in $K$.
	Then 
	$$\Big|\int_{B_{\frac{r}{n}}}g_x(y)dy\Big|\leq\frac{C}{n^{2-\alpha}}.$$
	We need to estimate
	\begin{equation}\label{LnL-1}
		I:=\sum_{y\in A_{\frac{r}{n},M}^{::}\cup\partial A_{\frac{r}{n},M}^{::}}\int_{[-\frac{1}{2n},\frac{1}{2n}]^2}\tilde{g}_{x,y}(z)dz.
	\end{equation}
	We use the second order Taylor's theorem in every argument of the sum and write
	$$\tilde{g}_{x,y}(z)=\sum_{|\beta|=2}z^\beta R_\beta(z),$$
	where $R_\beta(z)=\frac{|\beta|}{\beta!}\int_0^1(1-t)D^\beta g_{x,y}(tz)dt$. Since $f$ is twice differentiable, there exists a positive constant $C$ which only depends on the first and second partial derivatives of $f$ in $K+B_M$ and in the parameter $\alpha$, such that for all $z\in[-\frac{1}{2n},\frac{1}{2n}]^2$, it holds that 
	$$|R_\beta(z)|\leq C\Big(\frac{1}{|y+z|^{2+\alpha}}+\frac{1}{|y-z|^{2+\alpha}}\Big).$$
	Notice that $|z|\leq\frac{\sqrt{2}}{2n}$ if $z\in[-\frac{1}{2n},\frac{1}{2n}]^2$, since $y\in\frac{1}{n}\bb Z^2\setminus\{0\}$, we have (for a bigger constant $C$)
	$$|R_\beta(z)|\leq \frac{C}{(|y|-\frac{\sqrt{2}}{2n})^{2+\alpha}}.$$
	We conclude
	\begin{eqnarray*}
		|I|&\leq&\frac{C}{n^4}\sum_{y\in A_{\frac{r}{n},M}^{::}\cup\partial A_{\frac{r}{n},M}^{::}}\frac{1}{(|y|-\frac{\sqrt{2}}{2n})^{2+\alpha}}\\
		&\leq&\frac{C}{n^{2-\alpha}}\sum_{y\in\bb{Z}^2\setminus B_r}\frac{1}{(|y|-\frac{\sqrt{2}}{2})^{2+\alpha}}\\
		&\leq&\frac{C}{n^{2-\alpha}}.
	\end{eqnarray*}
	This finishes our proof. 
\end{proof}

\begin{proof}[Proof of lemma \ref{inversion-smooth-lemma}]
	
	Note that
	$$L_M(G_M*\phi)(x)=\int_{\bb{R}^2}G_M(y)L_M\phi(x-y)dy.$$
	By the symmetry of $L_M$, we have
	$$L_M(G_M*\phi)(x)=\int_{\bb{R}^2}L_MG_M(x-y)\phi(y)dy
	=\int_{\bb{R}^2}F_x(y)dy,$$
	where 
	\begin{equation}\label{greeninv1}
		F_x(y)=\frac{c\phi(y)}{2(2\pi)^2}\int_{B_M}\int_{\bb{R}^2}\frac{e^{i(\theta\cdot x-y+z)}+e^{i(\theta\cdot x-y+z)}-2e^{i(\theta\cdot x-y)}}{|z|^{2+\alpha}\psi_M(\theta)}d\theta dz.
	\end{equation}
	Since the integrand in \eqref{greeninv1} is not absolutely integrable, we are not able to apply Fubini's theorem directly. To overcome this difficulty we make the following trick,
	$$F_x(y)=\lim_{t\rightarrow0}\frac{c\phi(y)}{2(2\pi)^2}\int_{B_M}\int_{\bb{R}^2}\frac{e^{i(\theta\cdot x-y+z)}+e^{i(\theta\cdot x-y-z)}-2e^{i(\theta\cdot x-y)}}{|z|^{2+\alpha}\psi_M(\theta)}e^{-t|\theta|^2}d\theta dz.$$
	Now we use Fubini's theorem and integrate first on $z$. By the definition of $\psi_M$ we have
	$$F_x(y)=\lim_{t\rightarrow0}\frac{-\phi(y)}{(2\pi)^2}\int_{\bb{R}^2}e^{i(\theta\cdot x-y)}e^{-t|\theta|^2}d\theta.$$
	Hence, using Fubini's theorem again
	\begin{eqnarray*}
		\int_{\bb{R}^2}\phi(y)L_MG_M(x-y)dy&=&\lim_{t\rightarrow0}\frac{-1}{(2\pi)^2}\int_{\bb{R}^2}\int_{\bb{R}^2}\phi(y)e^{i(\theta\cdot x-y)}e^{-t|\theta|^2}dy d\theta\\ 
		&=&\lim_{t\rightarrow0}\frac{-1}{(2\pi)}\int_{\bb{R}^2}\hat{\phi}(\theta)e^{i(\theta\cdot x)}e^{-t|\theta|^2}d\theta\\
		&=&\frac{-1}{(2\pi)}\int_{\bb{R}^2}\hat{\phi}(\theta)e^{i(\theta\cdot x)}d\theta\\
		&=&-\phi(x).
	\end{eqnarray*}
	Since $\phi\in C_c^\infty$, in particular, $\hat{\phi}\in L^1$. This fact justifies the last two lines above.
\end{proof}

\begin{proof}[Proof of Lemma \ref{estimateG1}]
	Before starting the proof, a little note about the constants may be useful. 
	\begin{remark}\label{constants}
		Recall that the constant $c$ is defined by $c=\lim c_n$. Let us denote by $c_\alpha$ and $\tilde{c}_\alpha$, the constant satisfying
		\begin{eqnarray*}
			c_\alpha|\theta|^\alpha&=&c\int_{\bb{R}^2}\frac{1-\cos(\theta\cdot x)}{|x|^{2+\alpha}}dx\\
			\text{and~~~~~~}\frac{\tilde{c}_\alpha}{|\theta|^{2-\alpha}}&=&\int_{\bb{R}^2}\frac{\cos(\theta\cdot x)}{|x|^\alpha}dx.
		\end{eqnarray*}
		Standard computations show that $\tilde{c}_\alpha=\frac{\alpha^2c_\alpha}{c}$. 
	\end{remark}
	
	If $|\theta|<1$ we use Taylor's theorem and see that 
	$$\psi_1(\theta)=c\int_{B_1}\frac{1-\cos(\theta\cdot y)}{|y|^{2+\alpha}}dy=c\int_{B_1}\frac{(\theta\cdot y)^2}{2|y|^{2+\alpha}}dy+|\theta|^4g_1(\theta)$$
	for a smooth function $g$ in $B_1$.
	
	Note that 
	$$\frac{c}{2}\int_{B_1}\frac{(\theta\cdot y)^2}{|y|^{2+\alpha}}dy=\frac{c|\theta|^2}{4}\int_{B_1}\frac{1}{|y|^\alpha}dy=\frac{\sigma_1^2|\theta|^2}{4}.$$
	Then for $\theta\in B_1$
	\begin{equation}\label{psi1-1}
		\psi_1(\theta)=\frac{\sigma_1^2}{4}|\theta|^2+|\theta|^4g_1(\theta).
	\end{equation}
	Since $\psi_1(\theta)=0$ only if $\theta=0$, we easily see that there exists a smooth function $g_2$ in $B_1$ such that
	\begin{equation}\label{psi1inv-1}
		\frac{1}{\psi_1(\theta)}=\frac{4}{\sigma_1^2|\theta|^2}+g_2(\theta).
	\end{equation}
	Now let us suppose $|\theta|\geq 1$. Recalling from \ref{constants} the definition of the constant $c_\alpha$, we write 
	\begin{equation}\label{psi1-2}
		\psi_1(\theta)=c_\alpha|\theta|^\alpha-c\int_{B_1^c}\frac{1-\cos(\theta\cdot y)}{|y|^{2+\alpha}}dy.
	\end{equation}
	The function $h_1$ defined by $h_1(\theta):=c\int_{B_1^c}\frac{1-\cos(\theta\cdot y)}{|y|^{2+\alpha}}dy$ is clearly bounded. We can easily see that for $|\theta|>1$ there exists a bounded function $h_2$, such that
	\begin{equation}\label{psi1inv-2}
		\frac{1}{\psi_1(\theta)}=\frac{1}{c_\alpha|\theta|^\alpha}+\frac{h_2(\theta)}{|\theta|^{2\alpha}}.
	\end{equation}
	Furthermore the function $h_2$ satisfies $\frac{d}{d\theta_i}\frac{h_2(\theta)}{|\theta|^{2^\alpha}}=\frac{h_{2,i}(\theta)}{|\theta|^{2\alpha+1}}$ for $i=1,2$, and $\Delta(\frac{h_2(\theta)}{|\theta|^{2\alpha}})=\frac{h_3(\theta)}{|\theta|^{2\alpha+2}}$ for bounded functions $h_{2,i}$ and $h_3$ in $B_1^c$. The proof of this assertion lies in the computation of the first two derivatives of $\frac{1}{\psi_1(\theta)}-\frac{1}{c_\alpha|\theta|^\alpha}$ outside $B_1$.
	
	We can write $G_1$ as 
	\begin{equation}\label{G1tilde}
		G_1(x)=\frac{1}{(2\pi)^2}\widetilde{G}_1(x)-\frac{1}{(2\pi)^2}\widetilde{G}_1(x_0),
	\end{equation}
	where 
	\begin{equation}\label{Gtilde-def}
		\widetilde{G}_1(x)=\int_{B_1}\frac{\cos(\theta\cdot x)-1}{\psi_1(\theta)}d\theta+\int_{B_1^c}\frac{\cos(\theta\cdot x)}{\psi_1(\theta)}d\theta.
	\end{equation}
	Let us denote 
	\begin{eqnarray}
		f_1(x)&=&\int_{B_1}\frac{\cos(\theta\cdot x)-1}{|\theta|^2}d\theta,\label{f1}\\
		f_2(x)&=&\int_{B_1^c}\frac{\cos(\theta\cdot x)}{|\theta|^\alpha}d\theta,\label{f2}\\
		f_3(x)&=&\int_{B_1}(\cos(\theta\cdot x)-1)g_2(\theta)d\theta,\label{f3}\\
		f_4(x)&=&\int_{B_1^c}\frac{h_2(\theta)\cos(\theta\cdot x)}{|\theta|^{2\alpha}}d\theta\label{f4}.
	\end{eqnarray}
	By \eqref{psi1inv-1} and \eqref{psi1inv-2} we have 
	\begin{equation}\label{f1f2f3f4}
		\widetilde{G}_1(x)=\frac{4}{\sigma_1^2}f_1(x)+\frac{1}{c_\alpha}f_2(x)+f_3(x)+f_4(x).
	\end{equation}
	We split the estimation of $G_1$ in four steps, consisting in the estimation for $|x|<1$ and $|x|\geq1$ of the functions $f_1$, $f_2$, $f_3$, and $f_4$.
	
	\emph{Estimate of} $f_1$. Suppose first $|x|<1$, since in \eqref{f1} we are integrating a bounded function, $f_1$ is also bounded in the unitary ball. That is, there exists a constant $C>0$ such that
	\begin{equation}\label{step.1-estim-B1}
		|f_1(x)|<C\text{~for all~}|x|<1.
	\end{equation}
	Now let us assume $|x|\geq1$, we denote $\hat{x}=\frac{x}{|x|}$, and perform a change of variables to see that
	\begin{eqnarray*}
		f_1(x)&=&\int_{B_{|x|}}\frac{\cos(\theta\cdot \hat{x})-1}{|\theta|^2}d\theta\\
		&=&-\int_{B_{|x|}\setminus B_1}\frac{1}{|\theta|^2}d\theta+\int_{B_{|x|}\setminus B_1}\frac{\cos(\theta\cdot \hat{x})}{|\theta|^2}d\theta+\int_{B_1}\frac{\cos(\theta\cdot \hat{x})-1}{|\theta|^2}d\theta.
	\end{eqnarray*}
	The first term of the sum above can be explicitly computed and give us a logarithmic term
	$$-\int_{B_{|x|}\setminus B_1}\frac{1}{|\theta|^2}d\theta=-2\pi\log|x|.$$
	Notice that the last term is a constant, let us call it $C_1$. Let us investigate the second term, we call it $\tilde{h}(x)$, that is
	\begin{equation}\label{h-tilde}
		\tilde{h}(x)=\int_{B_{|x|}\setminus B_1}\frac{\cos(\theta\cdot \hat{x})}{|\theta|^2}d\theta.
	\end{equation}
	We integrate by parts in the following way:
	\begin{eqnarray*}
		\tilde{h}(x)&=&\int_{B_{|x|}\setminus B_1}\frac{\Delta(1-\cos(\theta\cdot \hat{x}))}{|\theta|^2}d\theta\\
		&=&\int_{B_{|x|}\setminus B_1}\frac{4(1-\cos(\theta\cdot\hat{x}))}{|\theta|^4}d\theta-\int_{\partial(B_{|x|}\setminus B_1)}\frac{1}{|\theta|^2}\frac{\partial}{\partial\nu}\cos(\theta\cdot\hat{x})dS\\
		&&~~~~~~~-\int_{\partial(B_{|x|}\setminus B_1)}(1-\cos(\theta\cdot\hat{x}))\frac{\partial}{\partial\nu}\frac{1}{|\theta|^2}dS,
	\end{eqnarray*}
	where $\nu$ represents the \emph{inward} pointing unit normal along $\partial(B_{|x|}\setminus B_1)$ (see \cite{Evans}). The integration by parts procedure allows us to conclude that $\tilde{h}(x)$ is bounded. 
	
	We have proved that there exists an explicitly computable constant $C_1$ and a bounded function $\tilde{h}(x)$ given by \eqref{h-tilde}, such that
	\begin{equation}\label{step.1-estim} 
		f_1(x)=-2\pi\log|x|+\tilde{h}(x)+C_1\text{~for all~}|x|\geq 1.
	\end{equation}
	
	\emph{Estimate of} $f_2$. Recall from \ref{constants} the definition of the constant $\tilde{c}_\alpha$. Let us write
	$$f_2(x)=\int_{B_1^c}\frac{\cos(\theta\cdot x)}{|\theta|^\alpha}d\theta=\frac{\tilde{c}_\alpha}{|x|^{2-\alpha}}-\int_{B_1}\frac{\cos(\theta\cdot x)}{|\theta|^\alpha}d\theta.$$
	Let us investigate the behavior of the function 
	$$\tilde{f}_2(x):=\int_{B_1}\frac{\cos(\theta\cdot x)}{|\theta|^\alpha}d\theta.$$
	We easily see that the function $\tilde{f}_2$ is bounded for $|x|<1$. Then there exists a constant $C>0$ independent of $x$ such that 
	\begin{equation}\label{step.2-estim-B1}
		\left|f_2(x)-\frac{\tilde{c}_\alpha}{|x|^{2-\alpha}}\right|<C\text{~for all~}|x|<1.
	\end{equation} 
	Let us study its asymptotic behavior for $x$ large. Suppose $|x|\geq1$, we write
	$$\tilde{f}_2(x)=\frac{1}{|x|^2}\int_{B_1}\frac{\Delta(1-\cos(\theta\cdot x))}{|\theta|^\alpha}d\theta.$$
	We perform an integration by parts bellow. We again denote by $\nu$ the inward pointing unit normal vector along the surface $\partial B_1$:
	\begin{eqnarray*}
		\tilde{f}_2(x)&=&\frac{\alpha^2}{|x|^2}\int_{B_1}\frac{1-\cos(\theta\cdot x)}{|\theta|^{2+\alpha}}d\theta-\frac{1}{|x|^2}\int_{\partial B_1}(1-\cos(\theta\cdot x))\frac{\partial}{\partial\nu}\frac{1}{|\theta|^\alpha}dS\\
		&&~~~~~~~~~~~~~~+\frac{1}{|x|^2}\int_{\partial B_1}\frac{1}{|\theta|^\alpha}\frac{\partial}{\partial\nu}(1-\cos(\theta\cdot x))dS.
	\end{eqnarray*} 
	The first two integrals above decay as $\frac{1}{|x|^2}$, and the last one decays as $\frac{1}{|x|}$. Since we are considering $|x|\geq1$ we obtain that there exists a constant $C>0$ which only depends on $\alpha$, such that $|\tilde{f}_2(x)|<\frac{C}{|x|}$.
	
	Since $f_2(x)=\frac{\tilde{c}_\alpha}{|x|^{2-\alpha}}-\tilde{f}_2(x)$, and since $2-\alpha<1$, we obtain that there exists a positive constant $C$ which depends on $\alpha$, such that
	\begin{equation}\label{step.2-estim}
		|f_2(x)|<\frac{C}{|x|^{2-\alpha}}\text{~for all~}|x|\geq 1.
	\end{equation}
	
	\emph{Estimate of} $f_3$. It is easy to see that $f_3(x)$ is bounded for $|x|<1$, that is, there exists a positive constant $C$ depending on $\alpha$, such that
	\begin{equation}\label{step.3-estim-B1}
		|f_3(x)|<C\text{~for all~}|x|<1.
	\end{equation}
	If $|x|\geq 1$, we have
	$$f_3(x)=\int_{B_1}g_2(\theta)\cos(\theta\cdot x) d\theta+C_2$$
	for a constant $C_2$. We perform an integration by parts on the first term on the right of the above formula, just like we did in the estimation of $f_2$ and conclude that there exists a positive constant $C$ depending only on $\alpha$, such that
	\begin{equation}\label{step.3-estim}
		|f_3(x)-C_2|<\frac{C}{|x|^2}\text{~for all~}|x|\geq 1.
	\end{equation}
	
	\emph{Estimate of} $f_4$. Since $\alpha>1$ and consequently $2\alpha>2$, we easily see that $f_4$ is bounded. Then there exists a positive constant $C$ depending on $\alpha$, such that
	\begin{equation}\label{step.4-estim-B1}
		|f_4(x)|<C\text{~for all~}|x|<1.
	\end{equation}
	Consider $|x|\geq 1$. Recall from the paragraph bellow \eqref{psi1inv-2} that there exist bounded functions $h_3$ and $h_{2,i}$ for $i=1,2$ such that $\Delta\frac{h_2(\theta)}{|\theta|^{2\alpha}}=\frac{h_3(\theta)}{|\theta|^{2\alpha+2}}$ and $\frac{d}{d\theta_i}\frac{h_2(\theta)}{|\theta|^{2\alpha}}=\frac{h_{2,i}(\theta)}{|\theta|^{2\alpha+1}}$.
	
	We write
	$$f_4(x)=\frac{1}{|x|^2}\int_{B_1^c}\frac{h_2(\theta)\Delta(1-\cos(\theta\cdot x))}{|\theta|^{2\alpha}}d\theta.$$
	Denoting by $\nu$ the inward pointing unit normal vector along $\partial B_1^c$, we integrate by parts and write
	\begin{eqnarray*}
		f_4(x)&=&\frac{1}{|x|^{2}}\int_{B_1^c}\frac{h_3(\theta)(1-\cos(\theta\cdot x))}{|\theta|^{2\alpha+2}}d\theta-\frac{1}{|x|^2}\int_{\partial B_1^c}(1-\cos(\theta\cdot x))\frac{\partial}{\partial\nu}\frac{h_2(\theta)}{|\theta|^{2\alpha}}dS\\
		&&~~~~~~~~~~~~~~~~~+\frac{1}{|x|^2}\int_{\partial B_1^c}\frac{h_2(\theta)}{|\theta|^{2\alpha}}\frac{\partial}{\partial\nu}(1-\cos(\theta\cdot x))dS.
	\end{eqnarray*}
	The first two integrals above decay as $\frac{1}{|x|^2}$ and the last one decays as $\frac{1}{|x|}$. Then there exists a positive constant $C$ independent of $x$ such that
	\begin{equation}\label{step.4-estim}
		|f_4(x)|<\frac{C}{|x|}\text{~for all~}|x|\geq 1.
	\end{equation} 
	
	We have proved that the functions $f_1$, $f_3$, and $f_4$ are bounded in the unitary ball $B_1$ (see \eqref{step.1-estim-B1}, \eqref{step.3-estim-B1}, and \eqref{step.4-estim-B1}). In \eqref{step.2-estim-B1} we proved that the function $f_2$ has a special behavior in $B_1$. By \eqref{f1f2f3f4}, we conclude that there exists a positive constant $C$ which only depends on $\alpha$, such that
	\begin{equation}\label{G1tilde-1}
		\left|\widetilde{G}_1(x)-\frac{\alpha^2}{c}\frac{1}{|x|^{2-\alpha}}\right|<C.
	\end{equation}
	See Remark \ref{constants} to justify the constants in the formula above.
	
	Now, in \eqref{step.2-estim}, \eqref{step.3-estim}, and \eqref{step.4-estim} we estimate the asymptotic behavior of $f_2$, $f_3$ and $f_4$. Since the lowest decay of the three functions is $\frac{1}{|x|^{2-\alpha}}$, which is the decay of $f_2$, we can say that there exists a constant $C>0$ depending on $\alpha$, such that 
	$$|f_2(x)|,|f_3(x)|,|f_4(x)|<\frac{C}{|x|^{2-\alpha}}.$$
	Finally notice that the function $f_1$ has a special behavior in $B_1^c$ (see \eqref{step.1-estim}). Then, by \eqref{f1f2f3f4} we conclude that, there exists an explicitly computable constant $\widetilde{C}$, and a positive constant $C$ independent of $x$, such that
	\begin{equation}\label{G1tilde-2}
		\left|\widetilde{G}_1(x)+\frac{8\pi}{\sigma_1^2}\log|x|-\frac{4}{\sigma_1^2}\tilde{h}(x)-\widetilde{C}\right|<\frac{C}{|x|^{2-\alpha}}\text{~for all~}|x|\geq1,
	\end{equation}
	where the function $\tilde{h}$ is given by \eqref{h-tilde}.
	
	By \eqref{G1tilde} and since we fixed $x_0$ so that $|x_0|=1$, we conclude the proof. 
\end{proof}

\begin{proof}[Proof of Proposition \ref{estimateGM}]
	Recall from \eqref{estimateG1-1} and \eqref{estimateG1-2} the definitions of the functions $g$, $h$, and $h_1$. By Proposition \ref{estimateG1} we have that, for $x\in B_M$, it holds that
	\begin{equation}\label{GM-eq1}
		G_1\left(\frac{x}{M}\right)=\frac{\alpha^2M^{2-\alpha}}{(2\pi)^2c}\frac{1}{|x|^{2-\alpha}}+g\left(\frac{x}{M}\right),
	\end{equation}
	and for $x\in B_M^c$ it holds that
	\begin{equation}\label{GM-eq2}
		G_1\left(\frac{x}{M}\right)=-\frac{2}{\pi\sigma_1^2}\log|x|+\frac{2}{\pi\sigma_1^2}\log M+h_1\left(\frac{x}{M}\right)+\frac{M_{2-\alpha}h(\frac{x}{M})}{|x|^{2-\alpha}}.
	\end{equation}
	Since we fixed $x_0$ so that $|x_0|=1$, we will always have $\frac{|x_0|}{M}<1$, then
	\begin{equation}\label{GM-eq3}
		G_1\left(\frac{x_0}{M}\right)=\frac{\alpha^2M^{2-\alpha}}{(2\pi)^2c}+g\left(\frac{x_0}{M}\right).
	\end{equation} 
	By \eqref{changegreen} we have
	$$G_M(x)=\frac{\alpha^2}{(2\pi)^2c}\frac{1}{|x|^{2-\alpha}}-\frac{\alpha^2}{(2\pi)^2c}+\frac{g(\tfrac{x}{M})-g(\tfrac{x_0}{M})}{M^{2-\alpha}}\text{~for~}x\in B_M$$
	and
	$$G_M(x)=-\frac{2}{\pi\sigma_M^2}\log|x|+\delta_M+\frac{h(\tfrac{x}{M})}{|x|^{2-\alpha}}+h_M(x)\text{~for~}x\in B_M^c$$
	where $\delta_M=-\frac{\alpha^2}{(2\pi)^2c}+\frac{2}{\pi\sigma_M^2}\log M-\frac{g(\tfrac{x_0}{M})}{M^{2-\alpha}}$, and $h_M$ is given by 
	$$h_M(x)=\frac{1}{\pi^2\sigma_M^2}\int_{B_{\frac{|x|}{M}}\setminus B_1}\frac{\cos(\theta\cdot\omega)}{|\theta|^2}d\theta$$ 
	for any unitary vector $\omega$.
	
	In the equations above we use that $\sigma_M^2=M^{2-\alpha}\sigma_1^2$; this can be checked by a simple change of variables. 
\end{proof}

\begin{proof}[Proof of Proposition \ref{GM-derivada}]
	By \eqref{changegreen}, it suffices to prove that the first partial derivatives of $G_1$ are in $L_{loc}^1$. 
	
	We recall from \eqref{G1tilde} that $G_1(x)=\frac{1}{(2\pi)^2}\widetilde{G}(x)-\frac{1}{(2\pi)^2}\widetilde{G}_1(x_0)$ for $\widetilde{G}_1$ defined in \eqref{Gtilde-def}. Then we only need to prove that the first partial derivatives of $\widetilde{G}_1$ are in $L_{loc}^1$.  
	
	By \eqref{f1f2f3f4}, it suffices that the first partial derivatives of the functions $f_1$, $f_2$, $f_3$, and $f_4$ defined in \eqref{f1}, \eqref{f2}, \eqref{f3}, and \eqref{f4} are in~$L_{loc}^1$. Standard arguments show this fact.
\end{proof}

\begin{proof}[Proof of Lemma \ref{obstacle-cont-concave}]
	Let us write
	\begin{eqnarray*}
		G_M*\rho(x)&=&\int_{\bb{R}^2} G_M(y)\rho(x-y)dy\\
		&=&\int_{B_M}G_M(y)\rho(x-y)dy+\int_{\bb{R}^2\setminus B_M}G_M(y)\rho(x-y)dy.
	\end{eqnarray*}
	We consider $x$ big enough so that $B_M\cap B(x,\supp{(\rho)})=\emptyset$, then we have
	$$G_M*\rho(x)=\int_{\bb{R}^2\setminus B_M}G_M(y)\rho(x-y)dy.$$
	Now we use Proposition \ref{estimateGM} to ensure that for $|y|>M$, there exists a constant $\delta_M$ (explicitly computable), and a bounded function $h_M$ in $B_M^c$ given by 
	\begin{equation}\label{h-M}
		h_M(y)=\frac{1}{\pi^2\sigma_M^2}\int_{B_{\frac{|y|}{M}}\setminus B_1}\frac{\cos(\theta\cdot \omega)}{|\theta|^2}d\theta,\text{~where $\omega$ is any unitary vector},
	\end{equation} 
	and a function $S(y)$ converging to zero as $y$ goes to infinity, such that 
	\begin{equation*}
		G_M(y)=-\tfrac{2}{\pi\sigma_M^2}\log|y|+h_M(y)+\delta_M+S(y).
	\end{equation*}
	Hence, we can write
	$$G_M*\rho(x)=\int_{\bb{R}^2\setminus B_M}\!\!\!\!S(y)\rho(x-y)dy+\int_{\supp(\rho)}\!\!(-\tfrac{2}{\pi\sigma_M^2}\log|x-y|+h_M(x-y)+\delta_M)\rho(y)dy.$$
	
	Now we differentiate twice, with respect to the variable $x$, under the integral sign. We note that all the second order derivatives of the logarithm function go to zero at infinity. we have to check that the function $h_M$ has the same property. We write $h_M$ in polar coordinates:
	$$h_M(x)=\frac{1}{\pi^2\sigma_M^2}\int_0^{2\pi}\int_1^{\frac{|x|}{M}}\frac{\cos(r\cos\beta)}{r}drd\beta.$$
	This change of variables facilitates the computation of the second order derivatives of $h_M$, so we can easily check that they converge to zero at infinity.
	
	We have proved that the second order derivatives of $G_M*\rho(x)$ converge to zero as $x$ goes to infinity. 
	
	Notice that the Hessian matrix of the function $-\frac{|x|^2}{\sigma_M^2}$ is equal to $-\frac{2}{\sigma_M^2}Id$ for all $x$. Then we can choose a ball $\Omega$ big enough, such that the Hessian matrix of the obstacle $\gamma$ is negative definite in $\bb{R}^2\setminus\Omega$. As a consequence, $\gamma$ is concave outside the set $\Omega$.
\end{proof}

\begin{proof}[Proof of Proposition \ref{odom-supp-cont}]
	We consider the concave envelope of the obstacle function $\gamma$, that is, the least concave function bigger than or equal to $\gamma$. We call this function $\Gamma$. 
	
	By Lemma \ref{obstacle-cont-concave}, there exists a ball $\Omega$ such that the function $\Gamma$ coincides with $\gamma$ outside $\Omega$. We want to prove that $\Gamma$ is continuous and superharmonic. 
	
	The concave envelope is always continuous. Since the obstacle $\gamma$ is twice differentiable, then $\Gamma$ is twice differentiable almost everywhere.  
	
	Fix $x\in\bb{R}^2$ such that $\Gamma$ is twice differentiable in $x$. Consider a tangent line $\ell_x$ to the graph of the function $\Gamma$ at the point $x$. By concavity, $\ell_x$ lies above the function $\Gamma$, then $\Gamma-\ell_x$ is nonpositive and attains its maximum at $x$ (where it is equal to zero). Hence $L_M(\Gamma-\ell_x)(x)\leq0$. Since the truncated fractional Laplacian of a linear function is equal to zero, we obtain that $L_M\Gamma(x)\leq0$.
	
	We repeat the same argument for all $x$ such that $\Gamma$ is twice differentiable and conclude that $L_M\Gamma(x)\leq0$ almost everywhere. This implies that for every smooth function $\phi$ with compact support and positive, it holds that
	$$\langle \Gamma,L_M\phi\rangle\leq0.$$
	We have proved that $\Gamma$ is continuous and superharmonic (recall definition \ref{def-alter-superh}). Hence we can ensure that $\Gamma\geq s$, where $s$ is the least superharmonic majorant defined in \eqref{majorantcontinuous}, and since $\Gamma$ coincides with $\gamma$ in $\Omega^c$, so does $s$. Since $u=s-\gamma$, we finish our proof.
\end{proof}

\section{Proof of Propositions of Section \ref{GREEN}}\label{convergence-green-proofs}

In this section we use a universal positive constant $C$ that only depends on $\alpha$, which may change from line to line. 

\begin{proof}[Proof of Lemma \ref{psiM-pequeno-estimate}]
	We write
	$$\frac{1}{c}\psi_M(\theta)-\widehat{\psi}_{M,n}(\theta)=\int_{B_M}\frac{1-\cos(\theta\cdot y)}{|y|^{2+\alpha}}dy-\frac{1}{n^2}\sum_{y\in B_M^{::}\cup\partial B_M^{::}}\frac{F_n(y)(1-\cos(\theta\cdot y))}{|y|^{2+\alpha}}.$$
	Since for $\theta\in B_{\frac{1}{M}}$ and $y\in B_M$ we have $|(\theta\cdot y)|\leq1$, by Taylor's theorem there exists a smooth function $g$ defined in the unitary ball that satisfies 
	$$1-\cos(\theta\cdot y)=\frac{(\theta\cdot y)^2}{2}+(\theta\cdot y)^4g(\theta\cdot y).$$
	An easy computation shows that 
	$$\frac{1}{2}\int_{B_M}\frac{(\theta\cdot y)^2}{|y|^{2+\alpha}}dy=\frac{|\theta|^2}{4}\int_{B_M}\frac{1}{|y|^\alpha}dy.$$
	Analogously 
	$$\frac{1}{2}\sum_{y\in B_M^{::}\cup\partial B_M^{::}}\frac{F_n(y)(\theta\cdot y)^2}{|y|^{2+\alpha}}=\frac{|\theta|^2}{4}\sum_{y\in B_M^{::}\cup\partial B_M^{::}}\frac{F_n(y)}{|y|^\alpha}.$$
	Then using the definition of $k_n$ (recall from \eqref{kn}), we have
	$$\frac{1}{c}\psi_M(\theta)-\widehat{\psi}_{M,n}(\theta)=\frac{k_n|\theta|^2}{n^{2-\alpha}}+A_1,$$
	where
	$$A_1:=\int_{B_M}\frac{(\theta\cdot y)^4g(\theta\cdot y)}{|y|^{2+\alpha}}dy-\frac{1}{n^2}\sum_{y\in B_M^{::}\cup\partial B_M^{::}}\frac{F_n(y)(\theta\cdot y)^4g(\theta\cdot y)}{|y|^{2+\alpha}}.$$
	Let us denote $f_{y,\theta}(z)=\frac{(\theta\cdot y+z)^4g(\theta\cdot y+z)}{|y+z|^{2+\alpha}}+\frac{(\theta\cdot y-z)^4g(\theta\cdot y-z)}{|y-z|^{2+\alpha}}-2\frac{(\theta\cdot y)^4g(\theta\cdot y)}{|y|^{2+\alpha}}$. We write
	\begin{equation}\label{4}
		A_1=\int_{B_{\frac{r}{n}}}\frac{(\theta\cdot z)^4g(\theta\cdot z)}{|z|^{2+\alpha}}dz+\frac{1}{2}\sum_{y\in A_{\frac{r}{n},M}^{::}\cup\partial A_{\frac{r}{n},M}^{::}}\int_{[-\frac{1}{2n},\frac{1}{2n}]^2}f_{y,\theta}(z)dz.
	\end{equation}
	We use a second-order Taylor expansion with remainder on the function $f_{y,\theta}(z)$; since $z\in [-\frac{1}{2n},\frac{1}{2n}]^2$ we have
	$$f_{y,\theta}(z)=\sum_{|\beta|=2}z^\beta R_{f_{y,\theta}}^\beta(z),$$
	where $R_{f_{y,\theta}}^\beta(z)=\frac{|\beta|}{\beta !}\int_0^1(1-t)^{|\beta|-1}D^\beta f_{y,\theta}(tz)dt$. We easily see that there exists a constant $C$ which depends on $\alpha$ such that
	$$|R_{f_y{y,\theta}}^\beta(z)|\leq C|\theta|^4\Big(\frac{1}{|y+z|^\alpha}+\frac{1}{|y-z|^\alpha}\Big).$$
	Since $z\in[-\frac{1}{2n},\frac{1}{2n}]^2$ we have $|z|\leq\frac{\sqrt{2}}{2n}$, and similarly as we did in the proof of Lemma \ref{LntoL}, we obtain 
	$$\frac{1}{2}\sum_{y\in A_{\frac{r}{n},M}^{n,::}\cup\partial A_{\frac{r}{n},M}^{n,::}}\int_{[-\frac{1}{2n},\frac{1}{2n}]^2}|f_y(z)|dz\leq\frac{C|\theta|^4}{n^2}.$$
	where the constant $C$ also depends only on $\alpha$.
	
	The first term in the right hand side of \eqref{4} is easily seen to be $\leq\frac{C|\theta|^4}{n^{4-\alpha}}$ for a constant $C$ depending only on $\alpha$. Since $4-\alpha>2$ we conclude 
	$$|A_1|\leq\frac{C|\theta|^4}{n^2}.$$
	We have just proved that there exists a positive constant $C$ which depends on $\alpha$ such that
	$$\Big|\frac{1}{c}\psi_M(\theta)-\widehat{\psi}_{M,n}(\theta)-\frac{k_n|\theta|^2}{n^{2-\alpha}}\Big|\leq \frac{C|\theta|^4}{n^2},$$ 
	which together with \eqref{psiMn-hat} concludes the proof of the lemma.
	
\end{proof}

\begin{proof}[Proof of Lemma \ref{psiM-grande-estimate}]
	We consider the function $\widetilde{F}$ in $\frac{1}{n}\bb{Z}^2$ defined as 
	$$\widetilde{F}(y)=n^2\times\mu(y^\square\cap B_{\frac{1}{|\theta|}}).$$ 
	We split the difference $\frac{1}{c}\psi_M(\theta)-\widehat{\psi}_{M,n}(\theta)$ into two parts and estimate each one separately. We write
	$$\frac{1}{c}\psi_M(\theta)-\widehat{\psi}_{M,n}(\theta)=B_1+B_2,$$
	where
	$$B_1:=\int_{B_{\frac{1}{|\theta|}}}\frac{1-\cos(\theta\cdot y)}{|y|^{2+\alpha}}dy-\frac{1}{n^2} \!\!\!\!\sum_{y\in B_{\frac{1}{|\theta|}}^{n,::}\cup\partial B_{\frac{1}{|\theta|}}^{n,::}} \!\!\!\!\!\!\frac{\widetilde{F}(y)(1-\cos(\theta\cdot y))}{|y|^{2+\alpha}}$$
	and
	$$B_2:=\int_{A_{\frac{1}{|\theta|},M}}\frac{1-\cos(\theta\cdot y)}{|y|^{2+\alpha}}dy-\frac{1}{n^2} \!\!\!\!\sum_{y\in A_{\frac{1}{|\theta|},M}^{n,::}\cup\partial A_{\frac{1}{|\theta|},M}^{n,::}} \!\!\!\!\!\!\frac{(F_n(y)-\widetilde{F}(y))(1-\cos(\theta\cdot y))}{|y|^{2+\alpha}},$$
	recalling that $A_{\frac{1}{|\theta|},M}=B_M\setminus B_{\frac{1}{|\theta|}}$.
	
	\emph{Estimate of} $B_1$. For $y\in B_{\frac{1}{|\theta|}}$, we can write
	$$1-\cos(\theta\cdot y)=\frac{(\theta\cdot y)^2}{2}+(\theta\cdot y)^4g(\theta\cdot y)$$
	for a smooth function $g$ defined in the unitary ball. Hence, we write
	\begin{equation}\label{8}
		B_1=B_{1,1}+B_{1,2},
	\end{equation}
	where
	$$B_{1,1}=\frac{|\theta|^2}{4}\Big(\int_{B_{\frac{1}{|\theta|}}}\frac{1}{|y|^\alpha}dy-\frac{1}{n^2}\sum_{y\in B_{\frac{1}{|\theta|}}^{n,::}\cup\partial B_{\frac{1}{|\theta|}}^{n,::}}\frac{\widetilde{F}(y)}{|y|^\alpha}\Big)$$
	and
	$$B_{1,2}=\int_{B_{\frac{1}{|\theta|}}}\frac{(\theta\cdot y)^4g(\theta\cdot y)}{|y|^{2+\alpha}}dy-\frac{1}{n^2}\sum_{y\in B_{\frac{1}{|\theta|}}^{n,::}\cup\partial B_{\frac{1}{|\theta|}}^{n,::}}\frac{\widetilde{F}(y)(\theta\cdot y)^4g(\theta\cdot y)}{|y|^{2+\alpha}}.$$
	We first estimate $B_{1,1}$. Note from the definition of $k_n$ (see \eqref{kn}) that
	$$B_{1,1}=\frac{k_n|\theta|^2}{n^{2-\alpha}}-\frac{|\theta|^2}{4}\Big(\int_{A_{\frac{1}{|\theta|},M}}\frac{1}{|y|^\alpha}dy-\frac{1}{n^2}\sum_{y\in \in A_{\frac{1}{|\theta|},M}^{::}\cup\partial A_{\frac{1}{|\theta|},M}^{n,::}}\frac{F_n(y)-\widetilde{F}(y)}{|y|^\alpha}\Big).$$
	Let us denote $f_y(z):=\frac{1}{2}(\frac{1}{|y+z|^\alpha}+\frac{1}{|y-z|^\alpha}-\frac{2}{|y|^\alpha})$. Then
	$$B_{1,1}=\frac{k_n|\theta|^2}{n^{2-\alpha}}-\frac{|\theta|^2}{4}\sum_{y\in A_{\frac{1}{|\theta|},M}^{n,::}\cup\partial A_{\frac{1}{|\theta|},M}^{n,::}}\int_{[-\frac{1}{2n},\frac{1}{2n}]^2}f_y(z)dz.$$ 
	We use Taylor's theorem and bound $f_y(z)$ by its second order derivatives just like we did in the proof of Lemma \ref{LntoL}. We obtain that there exists a constant $C$ which only depends on $\alpha$ such that
	$$\int_{[-\frac{1}{2n},\frac{1}{2n}]^2}|f_y(z)|dz\leq \frac{C|z|^4}{(|y|-\frac{\sqrt{2}}{2n})^{2+\alpha}}.$$
	Then
	$$\frac{|\theta|^2}{4}\sum_{y\in A_{\frac{1}{|\theta|},M}^{n,::}\cup\partial A_{\frac{1}{|\theta|},M}^{n,::}}\int_{[-\frac{1}{2n},\frac{1}{2n}]^2}|f_y(z)|dz\leq\frac{C|\theta|^{2+\alpha}}{n^2}.$$
	We have proved that there exists a constant $C$ which only depends on $\alpha$, such that
	\begin{equation}\label{12}
		\Big|B_{1,1}-\frac{k_n|\theta|^2}{n^{2-\alpha}}\Big|\leq\frac{C|\theta|^{2+\alpha}}{n^2}.
	\end{equation}
	Now let us estimate $B_{1,2}$. We proceed in the same way as in the estimation of $B_{1,1}$. Using the second order Taylor's theorem we obtain that there exists a constant $C$ depending only on $\alpha$, such that
	\begin{equation}\label{13}
		|B_{1,2}|\leq\frac{C|\theta|^4}{n^2}\int_{B_{\frac{1}{|\theta|}}}\frac{1}{|y|^\alpha}dy\leq \frac{C|\theta|^{2+\alpha}}{n^2},
	\end{equation}
	where the constant $C$ has changed in the last inequality. 
	
	Putting \eqref{12} and \eqref{13} into \eqref{8}, we obtain
	$$|B_1-\frac{k_n|\theta|^2}{n^{2-\alpha}}|\leq\frac{C|\theta|^{2+\alpha}}{n^2}.$$

	\emph{Estimate of} $B_2$.
	We repeat the same argument used in the estimate of $B_1$. That is, first we write the difference between the integral and the sum, as a sum of integrals over small squares. Then we symmetrize the arguments in the integrals and use the second order Taylor's theorem with remainder. We conclude that there exists a constant $C$ which only depends on $\alpha$, such that
	$$|B_2|\leq\frac{C|\theta|^{2+\alpha}}{n^2}.$$
	The estimations of $B_1$ and $B_2$ together give us
	$$\frac{1}{c}\psi(\theta)-\widehat{\psi}_n(\theta)-\frac{k_n|\theta|^2}{n^{2-\alpha}}=\frac{\tilde{r}_n(\theta)|\theta|^{2+\alpha}}{n^2},$$
	where $\tilde{r}_n$ is a uniformly bounded sequence of functions defined in $[-n\pi,n\pi]^2\setminus B_{\frac{1}{M}}$.
	
	Finally, by \eqref{psiMn-hat} we conclude the proof of the first part of Lemma \ref{psiM-grande-estimate}. The proof of (i) and (ii) follows the same idea used in the first part of the proof.
\end{proof}

\begin{proof}[Proof of Lemma \ref{GMntoGMtilde}]
	Let us introduce the notation
	\begin{eqnarray*}
		f_1(x)&=&\int_{B_{\frac{1}{M}}}\frac{\cos(\theta\cdot x)-1}{\psi_M(\theta)}d\theta,\\
		f_{1,n}(x)&=&\int_{B_{\frac{1}{M}}}\frac{\cos(\theta\cdot x)-1}{\psi_{M,n}(\theta)}d\theta,\\
		f_2(x)&=&\int_{[-n\pi,n\pi]^2\setminus B_{\frac{1}{M}}}\frac{\cos(\theta\cdot x)}{\psi_M(\theta)}d\theta,\\
		f_{2,n}(x)&=&\int_{[-n\pi,n\pi]^2\setminus B_{\frac{1}{M}}}\frac{\cos(\theta\cdot x)}{\psi_{M,n}(\theta)}d\theta,\\
		f_3(x)&=&\int_{\bb{R}^2\setminus [-n\pi,n\pi]^2}\frac{\cos(\theta\cdot x)}{\psi_M(\theta)}d\theta.
	\end{eqnarray*}
	Then 
	\begin{equation}\label{23}
		\frac{c}{c_n}\widetilde{G}_M(x)-\widetilde{G}_{M,n}(x)=\left(\frac{c}{c_n}f_1(x)-f_{1,n}(x)\right)+\left(\frac{c}{c_n}f_2(x)-f_{2,n}(x)\right)+\frac{c}{c_n}f_3(x).
	\end{equation}
	We split the estimation of $\frac{c}{c_n}\widetilde{G}_M-\widetilde{G}_{M,n}$ into three different steps, consisting in the estimation of $\frac{c}{c_n}f_1(x)-f_{1,n}(x)$, $\frac{c}{c_n}f_2(x)-f_{2,n}(x)$, and $\frac{c}{c_n}f_3(x)$, separately. 
	
	\emph{Estimate of} $\frac{c}{c_n}f_1(x)-f_{1,n}(x)$. By Lemma \ref{psiM-pequeno-estimate} we have that there exists a constant $C$ which only depends on $\alpha$, such that
	$$\Big|\frac{c}{c_n}f_1(x)-f_{1,n}(x)\Big|\leq \frac{C}{n^2}\int_{B_{\frac{1}{M}}}\frac{(1-\cos(\theta\cdot x))|\theta|^4}{\psi_{M,n}(\theta)\psi_M(\theta)}d\theta.$$
	Since $\psi_{M,n}$ and $\psi_M$ are both on the order of $|\theta|^2$ in $B_{\frac{1}{M}}$, we conclude
	\begin{equation}\label{G-1}
		\left|\frac{c}{c_n}f_1(x)-f_{1,n}(x)\right|\leq\frac{C}{n^2}.
	\end{equation}

	\emph{Estimate of} $\frac{c}{c_n}f_2(x)-f_{2,n}(x)$. In the region $[-n\pi,n\pi]^2\setminus B_{\frac{1}{M}}$, the functions $\psi_{M,n}$ and $\psi_M$ are both on the order $|\theta|^\alpha$.
	
	Denote $\phi_n(\theta):=\frac{c}{c_n\psi_M(\theta)}-\frac{1}{\psi_{M,n}(\theta)}$. Note that by Lemma \ref{psiM-grande-estimate} we have
	$$\phi_n(\theta)=\frac{c}{n^2c_n}\frac{r_n(\theta)|\theta|^{2+\alpha}}{\psi_M(\theta)\psi_{M,n}(\theta)}$$
	for a uniformly bounded sequence of function $r_n$ defined in $[-n\pi,n\pi]^2\setminus B_{\frac{1}{M}}$. We have
	\begin{eqnarray*}
		\frac{c}{c_n}f_2(x)-f_{2,n}(x)&=&\int_{[-n\pi,n\pi]^2\setminus B_{\frac{1}{M}}}\cos(\theta\cdot x)\phi_n(\theta)d\theta\\
		&=&\frac{-1}{|x|^2}\int_{[-n\pi,n\pi]^2\setminus B_{\frac{1}{M}}}\Delta(\cos(\theta\cdot x))\phi_n(\theta)d\theta.
	\end{eqnarray*}
	Integrating by parts, we obtain
	\begin{eqnarray*}
		\frac{c}{c_n}f_2(x)-f_{2,n}(x)&=&\frac{-1}{|x|^2}\int_{[-n\pi,n\pi]^2\setminus B_{\frac{1}{M}}}\cos(\theta\cdot x)\Delta\phi_n(\theta)d\theta\\
		&&~~~~~~+\frac{1}{|x|^2}\int_{\partial([-n\pi,n\pi]^2\setminus B_{\frac{1}{M}})}\cos(\theta\cdot x)\frac{\partial}{\partial\nu}\phi_n(\theta)dS\\
		&&~~~~~~-\frac{1}{|x|^2}\int_{\partial([-n\pi,n\pi]^2\setminus B_{\frac{1}{M}})}\phi_n(\theta)\frac{\partial}{\partial\nu}\cos(\theta\cdot x) dS,
	\end{eqnarray*}
	where $\nu$ represents the inward pointing unit normal vector along the surface\\ $\partial([-n\pi,n\pi]^2\setminus B_{\frac{1}{M}})$. By items $(i)$ and $(ii)$ of Lemma \ref{psiM-grande-estimate}, we have $|\Delta\phi_n(\theta)|\leq\frac{C}{n^2|\theta|^\alpha}$, then 
	\begin{eqnarray}\label{24}
		\nonumber\frac{1}{|x|^2}\Big|\int_{[-n\pi,n\pi]^2\setminus B_{\frac{1}{M}}}\cos(\theta\cdot x)\Delta\phi_n(\theta)d\theta\Big|&\leq&\frac{C}{n^2|x|^2}\int_{[-n\pi,n\pi]^2}\frac{1}{|\theta|^\alpha}d\theta\\
		&\leq&\frac{C}{n^\alpha|x|^2}.
	\end{eqnarray}
	Now let us estimate the boundary terms. Since $\psi_{M,n}$ is a periodic function with period $[-n\pi,n\pi]^2$, we have
	\begin{equation}\label{25}
		\int_{\partial [-n\pi,n\pi]^2}\cos(\theta\cdot x)\frac{\partial}{\partial\nu}\frac{1}{\psi_{M,n}(\theta)}dS=\int_{\partial [-n\pi,n\pi]^2}\frac{1}{\psi_{M,n}(\theta)}\frac{\partial}{\partial\nu}\cos(\theta\cdot x) dS=0.
	\end{equation}
	On the other hand it is easy to see that 
	\begin{equation}\label{26}
		\frac{1}{|x|^2}\Big|\int_{\partial B_{\frac{1}{M}}}\cos(\theta\cdot x)\frac{\partial}{\partial\nu}\phi_n(\theta)dS\Big|\leq\frac{C}{n^2|x|^2}
	\end{equation}
	and
	\begin{equation}\label{27}
		\frac{1}{|x|^2}\Big|\int_{\partial B_{\frac{1}{M}}}\phi_n(\theta)\frac{\partial}{\partial\nu}\cos(\theta\cdot x) dS\Big|\leq\frac{C}{n^2|x|}.
	\end{equation}
	Define 
	\begin{equation}\label{T}
		T_n(x):=\frac{c}{c_n|x|^2}\int_{\partial [-n\pi,n\pi]^2}\cos(\theta\cdot x)\frac{\partial}{\partial\nu}\frac{1}{\psi_M(\theta)}dS-\frac{1}{\psi_M(\theta)}\frac{\partial}{\partial\nu}\cos(\theta\cdot x) dS.
	\end{equation}
	Then, from \eqref{24}, \eqref{25}, \eqref{26}, and \eqref{27}, we conclude 
	\begin{equation}\label{G-2}
		\left|\frac{c}{c_n}f_2(x)-f_{2,n}(x)-T_n(x)\right|\leq \frac{C}{n^\alpha}\left(\frac{1}{|x|^2}+\frac{1}{|x|}\right).
	\end{equation}
	
	\emph{Estimate of} $\frac{c}{c_n}f_3(x)$. Note that
	\begin{eqnarray*}
		\frac{c}{c_n}f_3(x)&=&\frac{c}{c_n}\int_{\bb{R}^2\setminus [-n\pi,n\pi]^2}\frac{\cos(\theta\cdot x)}{\psi_M(\theta)}d\theta\\
		&=&\frac{-c}{c_n|x|^2}\int_{\bb{R}^2\setminus [-n\pi,n\pi]^2}\frac{\Delta(\cos(\theta\cdot x))}{\psi_M(\theta)}d\theta.
	\end{eqnarray*}
	Performing again an integration by parts, we obtain
	
	$$\frac{c}{c_n}f_3(x)=\frac{-c}{c_n|x|^2}\int_{\bb{R}^2\setminus [-n\pi,n\pi]^2}\cos(\theta\cdot x)\Delta\left(\frac{1}{\psi_M(\theta)}\right)d\theta-T_n(x),$$
	where $T_n$ was defined in \eqref{T}.
	
	Since $|\Delta\big(\frac{1}{\psi_M(\theta)}\big)|\leq\frac{C}{|\theta|^{2+\alpha}}$ on $\bb{R}^2\setminus B_{\frac{1}{M}}$ an, in particular, on $\bb{R}^2\setminus [-n\pi,n\pi]^2$, we have that
	$$\frac{1}{|x|^2}\left|\int_{\bb{R}^2\setminus [-n\pi,n\pi]^2}\cos(\theta\cdot x)\Delta\big(\frac{1}{\psi_M(\theta)}\big)d\theta\right|\leq\frac{C}{n^\alpha|x|^2}.$$
	Then
	\begin{equation}\label{G-3}
		\left|\frac{c}{c_n}f_3(x)+T_n(x)\right|\leq\frac{C}{n^\alpha|x|^2}.
	\end{equation}
	By \eqref{23}, \eqref{G-1}, \eqref{G-2}, and \eqref{G-3} we conclude the proof of Lemma \ref{GMntoGMtilde}
\end{proof}

\section*{Acknowledgments}
The authors would like to thank Charles Smart, Luis Caffarelli, and Luis Silvestre for fruitful discussions.


\begin{thebibliography}{10}
	
	\bibitem{BTW}
	P.~Bak, C.~Tang, and K.~Wiesenfeld.
	\newblock Self-organized criticality: An explanation of the 1/ \textit{f}
	noise.
	\newblock {\em Phys. Rev. Lett.}, 59 (1987), pp. 381-384.
	
	\bibitem{CaffarelliObstacle}
	L.~Caffarelli.
	\newblock The obstacle problem revisited.
	\newblock {\em Journal of Fourier Analysis and Applications}, 4 (1998), pp. 383-402.
	
	\bibitem{Evans}
	L.~Evans.
	\newblock {\em Partial Differential Equations}.
	\newblock Graduate studies in mathematics. American Mathematical Society, 2010.
	
	\bibitem{Friedman2010}
	A.~Friedman.
	\newblock {\em Variational Principles and Free-Boundary Problems}.
	\newblock Dover books on mathematics. Dover, Mineola, NY, 2010.
	
	\bibitem{LeGall}
	J.-F.~L. Gall and J.~Rosen.
	\newblock The range of stable random walks.
	\newblock {\em Ann. Probab.}, 19(2):650--705, 04 1991.
	
	\bibitem{LawlerIntersections}
	G.~Lawler.
	\newblock {\em Intersections of Random Walks}.
	\newblock Probability and its Applications. Birkh{\"a}user Boston, 1996.
	
	\bibitem{LawlerLimic}
	G.~Lawler and V.~Limic.
	\newblock {\em Random Walk: A Modern Introduction}.
	\newblock Cambridge Studies in Advanced Mathematics. Cambridge University
	Press, Cambridge, 2010.
	
	\bibitem{lawler1992}
	G.~F. Lawler, M.~Bramson, and D.~Griffeath.
	\newblock Internal diffusion limited aggregation.
	\newblock {\em Ann. Probab.}, 20 (1992), pp. 2117-2140.
	
	\bibitem{LevineSmart}
	L.~Levine, W.~Pegden, and C.~K. Smart.
	\newblock Apollonian structure in the Abelian sandpile, Geom. Funct. Anal., 26 (2016), pp. 306-336.
	
	\bibitem{Levine1}
	L.~Levine and Y.~Peres.
	\newblock Strong spherical asymptotics for rotor-router aggregation and the
	divisible sandpile.
	\newblock {\em Potential Analysis}, 30 (2009), pp. 1-27.
	
	\bibitem{Levine2}
	L.~Levine and Y.~Peres.
	\newblock Scaling limits for internal aggregation models with multiple sources.
	\newblock {\em Journal d'Analyse Mathématique}, 111 (2010), pp. 151-219.
	
	\bibitem{cyrille}
	C.~Lucas.
	\newblock The limiting shape for drifted internal diffusion limited aggregation
	is a true heat ball.
	\newblock {\em Probability Theory and Related Fields}, 159 (2014), pp. 197-235.
	
	\bibitem{smart}
	W.~Pegden and C.~K. Smart.
	\newblock Convergence of the Abelian sandpile.
	\newblock {\em Duke Math. J.}, 162 (2013), pp. 627-642.
	
	\bibitem{SadhuDhar}
	T.~Sadhu and D.~Dhar.
	\newblock Pattern formation in fast-growing sandpiles.
	\newblock {\em Phys. Rev. E}, 85 (2012), 021107.
	
	\bibitem{SilvestreTesis}
	L.~Silvestre.
	\newblock Regularity of the obstacle problem for a fractional power of the
	laplace operator.
	\newblock {\em Comm. on Pure and Appl. Math.}, 60 (2017), pp. 67-112.
	
\end{thebibliography}
\addcontentsline{toc}{section}{References}
\bibliographystyle{abbrv}

\end{document}